\documentclass[12pt]{article}

\usepackage{footmisc}
\usepackage{graphicx}
\usepackage{epsfig}
\usepackage{epstopdf}
\usepackage{verbatim}
\usepackage{fancyhdr}
\usepackage{amsthm}
\usepackage{amsmath}
\usepackage{amssymb}
\usepackage{float}
\allowdisplaybreaks[4]
\usepackage{bm}
\usepackage{eucal}
\usepackage{color}
\usepackage[dvipsnames]{xcolor}
\usepackage[titletoc]{appendix}
\usepackage[font=small,labelfont=bf]{caption}  
\usepackage[lofdepth,lotdepth]{subfig}
\usepackage{multirow}
\usepackage{tikz}
\usetikzlibrary{shapes.geometric}
\usetikzlibrary{plotmarks}
\usepackage{pgfplots}
\usepgfplotslibrary{groupplots}

\usepackage[ruled,linesnumbered]{algorithm2e}

\graphicspath{ {Figures/} }

\usepackage{tikz}  
\usetikzlibrary{positioning}
\usetikzlibrary{pgfplots.groupplots}

\theoremstyle{plain}
\newtheorem{theorem}{Theorem}[section]

\theoremstyle{plainNoItalics}

\newtheorem{remark}[theorem]{Remark}

\numberwithin{equation}{section}

\renewcommand{\div}{\nabla \cdot}
\newcommand{\R}{\mathbb{R}}
\newcommand{\F}{\mathbf{F}}
\newcommand{\G}{\mathbf{G}}
\newcommand{\Eflux}{\bm{\psi}}

\newcommand{\n}{\mathbf{n}}
\newcommand{\x}{\bm{x}}
\newcommand{\wf}{\widehat{F}}
\newcommand{\be}{\boldsymbol{\beta}}

\newcommand{\V}{\mathcal{V}_h^r}

\newcommand{\gm}{\gamma_1}
\newcommand{\ga}{\gamma_0}

\newcommand{\mt}{\mathcal{T}}
\newcommand{\me}{\mathcal{E}}

\newcommand{\Fh}{\mathcal{F}_h}
\newcommand{\Fmh}{\mathcal{F}^\mcM_h}
\newcommand{\FMh}{\mathcal{F}^M_h}

\newcommand{\mcM}{\mathcal{M}}
\newcommand{\mcT}{\mathcal{T}}

\newcommand{\bc}{\Gamma}

\newcommand{\lamGLF}{\lambda_\xi}
\newcommand{\diffc}{C_\xi}

\newcommand{\KM}{I_M}
\newcommand{\M}{M}

\newcommand{\stab}{s}  
\newcommand{\um}{u^{\mathtt{min}}}  
\newcommand{\uM}{u^{\mathtt{max}}}  

\newcommand{\uhm}{\bar{u}_{h}|_M} 
\newcommand{\uhmo}{\bar{u}_{h}^{n}|_M} 
\newcommand{\uhmn}{\bar{u}_{h}^{n+1}|_M}
\newcommand{\uhtildemn}{\bar{\tilde u}_h^{n+1}|_M }
\newcommand{\uhml}{\bar{u}_{h}^{n+1}|_M} 

\newcommand{\uhin}{\bar{u}_{h}^{n}|_{M_i}}
\newcommand{\uhfn}{\bar{u}_{h}^{n}|_{M_1}}
\newcommand{\uhln}{\bar{u}_{h}^{n}|_{M_L}}

\newcommand{\mm}{\mathtt{min}} 
\newcommand{\MM}{\mathtt{max}} 

\newcommand{\hiF}{\widehat{F}_{\xi}} 
\newcommand{\loF}{\widehat{F}^\mathcal{L}_{\xi}} 
\newcommand{\liF}{\wf_{\xi}^{\text{MPP}}}

\allowdisplaybreaks[4]

\begin{document}

\title{High-Order Discontinuous Cut Finite Element Methods for Scalar Hyperbolic Conservation Laws}
\author{Pei Fu\footnotemark[1], 
Gunilla Kreiss\footnotemark[2], 
Zelin Xin\footnotemark[1],  
Sara Zahedi\footnotemark[3],
}
\footnotetext[1]{School of Mathematics,  Nanjing University of Aeronautics and Astronautics, 211106 Nanjing, China. E-mail: fu.pei@nuaa.edu.cn, zelin\_xin@nuaa.edu.cn.}
\footnotetext[2]{Division of Scientific Computing, Department of Information Technology, Uppsala University, SE-75105 Uppsala, Sweden. E-mail:  gunilla.kreiss@it.uu.se.}
\footnotetext[3]{Department of Mathematics, KTH Royal Institute of Technology, SE-10044 Stockholm, Sweden. Email:  sara.zahedi@math.kth.se.}
\maketitle

\begin{abstract}
In this paper, we present a family of high-order cut finite element methods based on the discontinuous Galerkin (DG) framework for scalar hyperbolic conservation laws on complex domains. Building on our previous work, we develop a multidimensional formulation that combines macro-element stabilization with flux limiting to obtain a scheme that preserves the maximum principle and remains robust with respect to arbitrary boundary cuts of the background mesh. The physical domain is embedded in a regular background mesh, which may produce arbitrarily small cut cells. To avoid the severe time step restrictions typically associated with such cells, ghost penalty stabilization terms are added on interior facets of macro-elements. The resulting method exhibits stability and accuracy properties similar to those of standard DG methods on fitted meshes. An $L^2$-stability result is derived for the semi-discrete scheme under both periodic and inflow-outflow boundary conditions. To enforce the maximum principle and suppress nonphysical oscillations, we adapt limiter techniques from standard DG methods to the CutFEM setting by defining  limiting parameters on macro-elements. In particular, we present a macro-element-based parameterized flux limiter together with adaptations of the Zhang-Shu bound-preserving limiter and the Barth-Jespersen slope limiter. Numerical experiments in two and three spatial dimensions demonstrate optimal convergence orders, preservation of the maximum principle, and accurate shock capturing without spurious oscillations, even for challenging cut configurations involving very small-cut cell intersections.
\end{abstract}

\bigskip
\noindent {\bf Keywords:} Scalar hyperbolic conservation laws; Cut finite element method;  Discontinuous Galerkin method; Stability estimate; Maximum principle preserving
%

\section{Introduction}
\label{sec:intro}
In this work, we consider scalar hyperbolic conservation laws in multiple spatial dimensions, which arise in many applications in computational fluid dynamics. In such applications, the fluid is often confined to domains with complex geometries, making the development of accurate and robust numerical methods particularly challenging. 
Let $\Omega \subset \R^d$,  $d\geq 2$, be a bounded convex domain with a polygonal or smooth boundary $\bc$.  We consider the scalar hyperbolic conservation law 
\begin{align}
&u_t+\nabla\cdot\F(u)=0, & x\in\Omega,\ t\in(0,T] \label{eq:model2D}\\
&u(x,0)=u_0(x),            & x\in\Omega \label{eq:initialcond2D} 
\end{align}
with suitable boundary conditions on $\bc$. 
Here $u=u(\x,t)$ denotes the conserved variable, and $\F(u)=\bigl(f_1(u),f_2(u),\cdots, f_d(u)\bigr)^T$: $\R\to\R^d$ is the flux function. We consider both periodic and inflow boundary conditions.   
For the latter, the boundary $\bc$ is decomposed into  inflow $\bc^-$ and  outflow $\bc^+$ parts, 
\begin{align}\label{eq:def:boundryflux}
\bc^-=\{\x\in\bc,\F'(u)\cdot\n_\bc(\x)\leq 0\}, \quad \bc^+=\{\x\in\bc,\F'(u)\cdot\n_\bc(\x) > 0\},
\end{align}
where $\n_\bc$ denotes the outward unit normal. On the inflow boundary, we impose 
\begin{align}\label{eq:def:boundrycon}
u(\x,t)=g(\x,t), \,\x \in\bc^-.
\end{align}
We assume there exists an invariant interval  containing all relevant states such that the normal characteristic speed has a uniform sign on each boundary portion. Since solutions of \eqref{eq:model2D} may develop discontinuities in finite time, weak solutions are considered.  However, weak solutions are generally not unique, and an additional entropy condition is required to select the physically relevant solution. Specifically, an entropy solution satisfies
$
 U(u)_t + \div \Eflux(u) \le 0, 
 $
 for every convex entropy function $U(u)$, i.e., $U''(u)\ge 0$, where the associated entropy flux $\Eflux(u)$ is defined by $\Eflux'(u)=U'(u)\F'(u)$.   Entropy inequalities provide a natural framework for the stability analysis of numerical methods. For scalar conservation laws, the quadratic entropy $U(u)=u^2/2$ is commonly used in the analysis of discontinuous Galerkin (DG) methods~\cite{jiang1994cell,shu2009discontinuous}, as it leads naturally to $L^2$-stability estimates. 
For this entropy, the corresponding entropy flux can be expressed in terms of a primitive $\G$ of the physical flux $\F$, defined by $\G'(u)=\F(u)$ such that $\Eflux(u)=u\F(u)-\G(u),$ up to an additive constant. This representation,  used in~\cite{jiang1994cell} forms the basis of the $L^2$-stability analysis carried out in this work.  The numerical approximation of conservation laws on complex geometries has been studied extensively.  Examples of numerical methods include the finite volume based h-box method \cite{berger2003h,berger2012simplified}, flux redistribution methods \cite{colella2006cartesian},  the inverse Lax-Wendroff method \cite{tan2010inverse}, domain of dependence stabilization methods \cite{engwer2020stabilized,may2022dod}, stabilized discontinuous Galerkin (DG) methods \cite{engwer2020stabilized}, and the state redistribution method \cite{berger2021state,giuliani2022two}.

In recent years, Cut Finite Element Methods (CutFEM) \cite{burman2025cut} have gained significant attention as accurate and robust approaches for problems posed on complex geometries. By allowing the use of general computational meshes, CutFEM avoids the need for admissible body-fitted meshes, which can be difficult to generate in practice. In this framework, the discretization can be formulated on a Cartesian background mesh that is straightforward to construct, and standard finite element basis functions can still be employed on standard elements. However, the unfitted nature of the mesh may give rise to cut elements with only a very small portion of the element intersecting the physical domain. Such cut elements may introduce severe time-step restrictions in explicit schemes and lead to ill-conditioned linear systems. 

To address these difficulties, we adopt a weak stabilization approach in which stabilization terms, such as ghost penalty terms, are incorporated into the variational formulation in order to ensure stability and robustness \cite{burman2010ghost,schoeder2020high,gurkan2019stabilized,gurkan2020stabilized,fu2021high}. An alternative strategy, commonly used in the context of DG methods, is cell merging, where small cut elements are combined with neighboring elements having larger intersections with the domain \cite{Johansson2013,kummer2017extended,modisette2010toward,muller2017high,Qin2013}. 

In our earlier work~\cite{fu2024bound}, we considered the one-dimensional case and showed that limiter-based strategies from DG methods can be extended to the CutFEM framework by applying the limiters on macro-elements~\cite{LarZah23}. The extension to higher spatial dimensions is far from trivial from an analytical perspective, since the domain boundary may cut the mesh in an arbitrary manner and the resulting macro-elements can have complex geometries.

In this work, we address these challenges and develop a high-order unfitted discretization for scalar conservation laws in multiple spatial dimensions that satisfies a discrete maximum principle and remains robust in the presence of discontinuities.  We propose a Cut Finite Element Method within a DG framework \cite{ShuDG4,shu2009discontinuous}, combined with explicit Runge-Kutta time discretization.  In standard DG methods, limiters are typically applied as postprocessing steps after each Runge--Kutta stage. Bound-preserving limiters \cite{zhang2010maximum,Zhang2011MaximumPrincipleSatisfyingAP,10.1098/rspa.2011.0153} enforce admissible solution bounds, while slope limiters \cite{Zhang2011MaximumPrincipleSatisfyingAP,zhang2010positivity} are used to suppress oscillations in the presence of shocks.  

The analysis of bound-preserving limiters typically relies on two key assumptions: (i) the cell average satisfies the maximum principle, and (ii) the extrema of the polynomial solution can be computed \cite{liu1996nonoscillatory,zhang2010maximum}, or alternatively, a cell-average decomposition based on exact quadrature rules is available \cite{Zhang2011MaximumPrincipleSatisfyingAP,10.1098/rspa.2011.0153}. While these assumptions can be ensured for standard DG methods, they become problematic in CutFEM due to the presence of cut elements. In particular, preserving bounds for the cell average may require severe time-step restrictions depending on the cut configuration, which we aim to avoid. At the same time, computing exact extrema of polynomial solutions on general macro-elements is impractical.  These difficulties motivate the development of a macro-element-based flux limiting strategy capable of enforcing admissible mean values independently of the cut configuration.

Flux limiting, originally developed for finite difference methods \cite{zalesak1979fully} and later extended to finite volume methods \cite{xu2014parametrized,CHRISTLIEB2015334} and to DG methods \cite{xiong2015high}, modifies numerical fluxes in order to guarantee that values or cell averages remain admissible. Within our framework, the flux limiter guarantees that the mean value on each macro-element satisfies the maximum principle without requiring the computation of exact extrema of the polynomial solution or a cell-average decomposition on general macro-elements. Compared with standard bound-preserving limiters, the proposed flux-limiting strategy substantially relaxes the time-step restrictions associated with cut elements. However, it guarantees admissibility only for macro-element mean values and not for the polynomial solution itself.
Once the macro-element mean values are controlled, a bound-preserving limiter can be applied at the final stage of each Runge--Kutta step, in order to enforce the maximum principle for the polynomial solution. In this step, the minimum and maximum values are approximated using the solution values at quadrature points on the underlying elements forming the macro-element. This strategy significantly reduces both computational cost and implementation complexity on general unstructured meshes while maintaining stability and accuracy.  To suppress oscillations near discontinuities and improve robustness for strong shock problems, we adapt the Barth-Jespersen slope limiter \cite{barth1989design} to the CutFEM setting by defining its limiting parameters on macro-elements. This slope limiter is applied at every stage of each Runge--Kutta step in the presence of discontinuities.   

The main contributions of this work are the introduction of a macro-element-based flux limiting strategy for high-order CutFEM discretizations of scalar conservation laws, the adaptation of slope and bound-preserving limiters from standard DG methods to unfitted meshes, and the derivation of energy stability results for the proposed scheme. Numerical experiments in both two and three spatial dimensions demonstrate that the scheme preserves the maximum principle, achieves the expected convergence rates, and produces oscillation-free solutions in the presence of shocks/discontinuities.

The paper is organized as follows. In Section~\ref{sec:scheme}, we present the proposed numerical scheme, together with the flux-limiting strategy and the modifications of the bound-preserving and slope limiters for unfitted meshes. In Section~\ref{sec:stability}, we establish the $L^2$-energy stability of the numerical solution $u_h$ for both linear and nonlinear fluxes. We also show that the proposed scheme with flux limiting guarantees that the macro-element mean values satisfy the maximum principle. Finally, in Section~\ref{sec:num}, numerical experiments demonstrate the expected convergence rates and confirm that the proposed method preserves the maximum principle.

\newcommand\nn{\boldsymbol{n}}
\newcommand\xx{\boldsymbol{x}}
\newcommand\hl{\mathcal{l}}
\newcommand\tf{\widetilde{F}}
\newcommand{\todo}[1]{{\color{blue}\bf#1}}
\newcommand{\xzlmodify}[1]{{\color{purple}\bf#1}}

\section{The numerical scheme}\label{sec:scheme}
This section presents the unfitted discretization for the scalar hyperbolic conservation law \eqref{eq:model2D}, including the mesh construction, finite element spaces, weak formulation, time discretization, and limiting procedures.

\subsection{Mesh and finite element space}
\label{sec:notation}
Let $\Omega$ be embedded in a background Cartesian mesh $\mt_{B,h}$ consisting of quasi-uniform, shape-regular triangles or rectangles. Define the active mesh $\mt_h$ and the corresponding active domain $\Omega_{\mt_h}$ by 
\begin{align}\label{eq:def:mesh}
\mt_h=\{K\in\mt_{B,h}: K\cap\Omega\neq\emptyset\}, \quad \Omega_{\mt_h}=\bigcup_{K\in\mt_h} K.
\end{align}
We also define the set $\me_h$ of all interior element facets (edges in 2D and faces in 3D) in $\mt_h$:
\begin{align}
\me_h=\{e=K_1\cap K_2:K_1,K_2\in\mt_h \text{ and } K_1\neq K_2\},
\end{align}
and introduce the sets of elements that are cut by the domain boundary $\bc$ (the cut elements):
\begin{align}
\label{eq:def:my_gamma}
\mt_\bc&=\{K\in\mt_h:K\cap\bc\neq \emptyset\}.
\end{align}

Let $\V$ be the finite element space on the active mesh $\mt_h$ consisting of piecewise polynomials of degree at most $r$:
\begin{align}\label{eq:def:space:V}
\V=\{v_h: v_h|_K\in P^r(K), \forall K\in\mt_h\},
\end{align}
where $P^r(K)$ is the space of polynomials of degree at most $r$ in the element $K$. Note that functions in $\V$ are discontinuous across interior facets of the active mesh. Furthermore, in the unfitted setting both facets and elements of the active mesh may be cut by the boundary $\Gamma$. 
We introduce separate notation for interior, boundary interfaces and complete set of interfaces:
\[
\mathcal F_h^e
=
\{e\cap\Omega : e\in \me_h\},
\quad
\mathcal F_h^\Gamma
=
\{K\cap\Gamma : K\in \mt_\bc\}, 
\quad \mathcal F_h
=
\mathcal F_h^e
\cup
\mathcal F_h^\Gamma.
\]

We introduce a unified notation for averages and jumps on all interfaces (both interior and boundary).  
For a quantity $w$ with traces $w^\pm$ on an interface $\xi\in\mathcal F_h$, we define the average and jump by
\begin{align}\label{eq:def:jump}
\{w\}_\xi = \frac{1}{2}\left(w^- + w^+\right), \quad [w]_\xi = w^+ - w^-.
\end{align} 
Here, $w^\pm$ denote traces with respect to the unit normal $\n_\xi$, where $\n_\xi$ is oriented from the ``$-$" side to the ``$+$" side.  On interior facets $\xi\in\mathcal F_h^e$, the traces are taken from the two neighboring elements according to the chosen orientation of $\n_\xi$. On boundary interfaces $\xi = K\cap\Gamma$, we set $\mathbf{n}_\xi=\mathbf{n}_\Gamma$.  For test functions $v_h\in\V$, the interior trace is denoted by $v_h^-$, while the exterior trace is defined by $v_h^+=0$. Consequently, 
$
[v_h]_\xi = -v_h^-.$ For the numerical solution $u_h$, the exterior trace is determined by the boundary condition. On inflow boundaries $\bc^-$, we set $u_h^+=g$, whereas on outflow boundaries $\bc^+$ we take $u_h^+=u_h^-$.

\subsubsection{Macro-elements}\label{sec:macroel}
For stabilization and limiting purposes, we introduce a macro-element partition $\mcM_h$ of the active domain 
$\Omega_{\mt_h}$ following  \cite{larson2021conservative}.  

We define the set of elements with a large $\Omega$-intersection by
\begin{equation}\label{eq:largeel}
\mt_h^L = \left\{ K \in \mt_h : \frac{|K \cap \Omega|}{|K|} \geq \delta \right\}.
\end{equation}
Here $0 < \delta \leq 1$ is a constant independent of both the element and the mesh size $h$. In the numerical simulations, we take $\delta = 0.2$. We assume that $h$ is sufficiently small so that $\mt_h^L$ is nonempty. Elements in $\mt_h^L$ are referred to as root cells. To each root cell, we associate a macro-element $M$ defined by
\begin{align}\label{eq:def:macro}
M = \bigcup_{j \in N_M} K_j, \quad K_j \in \mt_h,
\end{align}
where $N_M$ denotes the index set of elements included in $M$. Each macro-element consists either of a single root cell or of a root cell together with neighboring elements (reachable through a uniformly bounded number of interior facets) that have small intersections with $\Omega$. When the geometry is sufficiently smooth and well resolved, the macro-elements can be constructed so that they are comparable in size and consist of only a small number of neighboring elements; see Algorithm~1 \cite{larson2021conservative}.

For each nontrivial macro-element $M$ (i.e., consisting of more than one element), we define  by $\me_M$ the set of its interior facets. We then define
\begin{equation}
\me_{\mathrm{stab}} = \bigcup_{M \in \mcM_h} \me_M. 
\end{equation}
The set $\me_{\mathrm{stab}}$ collects all interior facets of nontrivial macro-elements and the stabilization terms introduced below will be evaluated on these facets. Note that $\me_{\mathrm{stab}}$ is a subset of the set of interior facets belonging to cut elements, since it excludes facets shared by distinct macro-elements.

\subsection{Semi-discrete unfitted DG method}\label{sec:semidiscrete}
In this section, we present a semi-discrete cut DG scheme for \eqref{eq:model2D}. Multiplying \eqref{eq:model2D} by a test function $v_h \in \V$, integrating by parts over each element in the physical domain, and adding stabilization terms yields the following semi-discrete formulation: find $u_h(\cdot,t) \in \V$ such that, for a.e.\ $t \in (0,T]$ and all $v_h \in \V$,
\begin{align}
&\left( u_{h,t},v_h \right)_{\Omega}+\gm \stab_{h,1}(u_{h,t},v_h)+a_h(u_h,v_h)+\ga \stab_{h,0}(u_h,v_h)=0,\label{eq:def:scheme}\\
&a_h(u_h,v_h)
=
-\sum_{K\in\mt_h} (\F(u_h), \nabla v_h)_{K\cap\Omega}
-\sum_{\xi \in \mathcal{F}_h}
\int_\xi \widehat{F}_{\xi}(u_h)\, [v_h]_\xi\,ds.\label{scheme:2D:ah}
\end{align}
The initial condition is defined by: find $u_h(0)=u_h(\cdot,0) \in \V$ such that
\begin{equation}\label{eq:def:initial}
(u_h(0),v_h)_{\Omega}+\gm \stab_{h,1}(u_h(0),v_h)=(u_0,v_h)_{\Omega}, \quad \forall v_h \in \V.
\end{equation}

For $m \in \{0,1\}$, the stabilization bilinear form $s_{h,m}$ is defined by
\begin{equation}\label{stable:2D}
\stab_{h,m}(u_h, v_h)= 
\sum_{e \in \me_{\mathrm{stab}}}  \alpha_M^{1-m} \sum_{k=0}^{r}  \omega_k h^{2k+m}\left< [\partial_n^k u_h]_e,[\partial_n^{k} v_h]_{e} \right>_e,
\end{equation} 
where $\omega_{k}=1/((2k+1)(k!)^2)$, and $\alpha_M$ denotes the maximum absolute value of the Jacobian $\partial(\F(u)\cdot \mathbf{n})/\partial u$, evaluated over all interfaces contained in the macro-element $\M$. The stabilization constants $\gm$ and $\ga$ are positive parameters. In the computations presented in this paper, we choose $\gm=0.25$ and $\ga=0.5$. We have not optimized these values and other choices also yield satisfactory numerical results. 

The sign of the interface term in \eqref{scheme:2D:ah} follows from the jump convention \eqref{eq:def:jump}. Indeed, summing the element boundary contributions arising from integration by parts 
$\sum_{K\in\mt_h}
\int_{\partial(K\cap\Omega)}
\F(u_h)\cdot\n_K\,v_h\,ds
$
, and expressing them in terms of traces on the interfaces $\xi\in\mathcal F_h$, while replacing the physical normal flux by a numerical flux $\widehat F_\xi$, yields the interface contribution
$-\sum_{\xi\in\mathcal F_h}
\int_\xi \widehat F_\xi(u_h)\,[v_h]_\xi\,ds.
$  
This motivates the definition of the bilinear form $a_h$.

The numerical flux on $\xi$ is defined by
\begin{equation}\label{eq:numericalFlux}
\widehat{F}_\xi(u_h) := \widehat{F}\big(u_h^-, u_h^+;\, \mathbf{n}_\xi \big),
\end{equation}
where as before $u_h^-$ and $u_h^+$ denote the interior and exterior states with respect to $\mathbf{n}_\xi$. 
As in standard DG methods \cite{ShuDG4}, we assume that $\widehat{F}(u,v;\mathbf{n})$ is a consistent, two-point Lipschitz flux satisfying:
\begin{enumerate}
\item[(F.1)] (Consistency) $\widehat{F}(u,u;\mathbf{n}) = \mathbf{F}(u)\cdot \mathbf{n}$,
\item[(F.2)] (Monotonicity) $\widehat{F}(u,v;\mathbf{n})$ is nondecreasing in $u$ and nonincreasing in $v$,
\item[(F.3)] (Lipschitz continuity) $\widehat{F}(u,v;\mathbf{n})$ is Lipschitz continuous,
\item[(F.4)] (Conservation) $\widehat{F}(u,v;\mathbf{n}) = -\widehat{F}(v,u;-\mathbf{n})$.
\end{enumerate}

We use the generalized Lax--Friedrichs flux as in \cite{li2020discontinuous}, which we apply on all interfaces, both interior facets and boundary interfaces.  The generalized Lax--Friedrichs flux is defined by   
\begin{align}\label{eq:def:LF}
\widehat{F}_\xi(u) =\widehat{F}(u^-,u^+;\n_\xi)
= \{\F(u)\}_\xi \cdot \n_\xi - \frac{\lamGLF}{2}[u]_\xi,
\end{align}
with 
$\lamGLF=\diffc \alpha_\xi,$ 
where 
$\alpha_\xi = \max_{u\in I_\xi} |\F'(u)\cdot\n_\xi|,$  
and $I_\xi$ is the interval between the two states used in the evaluation of the numerical flux on the facet $\xi$, i.e., $I_\xi= [\min(u^-,u^+),\max(u^-,u^+) ]$.  The parameter $\alpha_\xi$ may be chosen either locally or globally. For a high-order scheme, we use a local choice, whereas a global choice (same constant $\alpha_\xi$ on all facets) is used in the low-order scheme associated with the flux limiter introduced in Section \ref{sec:fluxlim}.  On interior facets and outflow boundary facets we set $\diffc=1$, recovering the standard Lax--Friedrichs flux. On inflow boundary facets $\xi\in\Gamma^-$, we need $C_\xi>$1 to prove stability. 
The numerical flux at the physical inflow boundary $\bc^-$ is then
\begin{equation}
\wf\big(w^-, g;\, \n_\xi \big)=\frac{1}{2} \left( F(w^-)+F(g) \right)\cdot\n_\xi
-
\frac{\lamGLF}{2} \left( g-w^- \right), 
\label{eq:boundary-gllf-flux}
\end{equation}
with $\lamGLF=\diffc \max_{u\in[\min\{w^-,g\},\,\max\{w^-,g\}]} |f'(u)|$ and $\diffc>1$. From  Lemma~2.1 in  \cite{li2020discontinuous}, this flux satisfies (F.1)-(F.4). Hence, we use the standard Lax--Friedrichs flux on interior facets and a more diffusive generalized Lax--Friedrichs flux on inflow boundary facets. The additional diffusion provides a penalty term that allows the mixed term generated by nonhomogeneous boundary data to be absorbed in the stability analysis. 
In the numerical simulations we did not observe any differences using the standard Lax--Friedrich flux (i.e., $C_\xi=1$ on all facets) compared to using the generalized Lax-Friedrich flux on the inflow boundary with $\diffc=2$. 

Using the definitions of average and jump on interfaces $\xi \in \mathcal{F}_h$, the bilinear form $a_h$ can be written as
\begin{align*}
a_h(w,v)=
-\sum_{K\in\mt_h}(\F(w), \nabla v)_{K\cap\Omega}
-\sum_{\xi \in \mathcal{F}_h}
\int_{\xi}
\left(\{\F(w)\}_\xi \cdot \mathbf{n}_\xi  - \frac{\lamGLF}{2}[w]_\xi\right)[v]_\xi\,ds.
\end{align*}

\subsection{Basic fully-discrete scheme}
We define the stabilized $L^2$-inner product
\begin{equation}
(u_h, v_h)_s := (u_h, v_h)_\Omega + \gm \stab_{h,1}(u_h, v_h),  
\quad \forall v_h \in \V.
\end{equation}
With this notation, the semi-discrete formulation \eqref{eq:def:scheme}  can be written equivalently as: find $u_h(\cdot,t) \in \V$ such that, for a.e. $t \in (0,T]$,
\begin{align}
(u_{h,t}, v_h)_s = -\left(a_h(u_h,v_h)+\ga \stab_{h,0}(u_h,v_h)\right), 
\quad \forall v_h \in \V. 
\end{align}

Applying the forward Euler method for the time discretisation, and denoting by $u_h^{n}$ the numerical solution at time $t=t_n$, for $n=0, 1, \cdots, N_T$, 
 the fully discrete scheme reads: given $u_h^{n} \in \V$ find $u_h^{n+1} \in \V$ such that
\begin{equation}\label{eq:def:scheme:fullyscheme}
(u_h^{n+1}, v_h)_s
= (u_h^n, v_h)_s
- \Delta t \left( a_h(u_h^n,v_h)+\ga \stab_{h,0}(u_h^n,v_h) \right),  
\quad \forall v_h \in \V,
\end{equation}
with the initial solution given by: find $u_h^{0} \in \V$ such that
\begin{equation}
(u_{h}^0, v_h)_s=(u_0,v_h)_{\Omega}, \quad \forall v_h \in \V.
\end{equation}
Here $\Delta t=t_{n+1}-t_n$ denotes the time-step and $N_T$ is the total number of time steps.

For the numerical experiments we employ the third-order strong stability preserving Runge--Kutta (SSP-RK) method \cite{gottlieb2001strong}. This method can be written as a convex combination of forward Euler updates:
\begin{align}
u_h^{(1)} &= u_h^n + \Delta t L(u_h^n), \notag\\
u_h^{(2)} &= u_h^n + \frac{1}{4}\Delta t \left(  L(u_h^n) + L(u_h^{(1)}) \right), \notag\\
u_h^{n+1} &= u_h^n + \frac{1}{6}\Delta t \left( L(u_h^n)+ L(u_h^{(1)})  + 4L(u_h^{(2)})\right), \label{eq:def:sspRK3}
\end{align}
applied to the system $u_{h,t} = L(u_h)$, where $L(u_h)$ is defined by
\begin{equation}
(L(u_h), v_h)_s = - a_h(u_h, v_h) - \ga \stab_{h,0}(u_h, v_h),
\quad \forall v_h \in \V.
\end{equation}

\subsection{Flux-limiting}\label{sec:fluxlim}
Entropy solutions to scalar conservation laws \eqref{eq:model2D}  satisfy a maximum principle.   In the case of a Cauchy problem, this reads
\begin{equation}
 \text{if } u_0(\x) \in [\um, \uM], \text{ then } u(\x,t) \in [\um, \uM], \text { for } t>0, \label{eq:maxprinciple}
\end{equation}
where $\um$ and $\uM$ denote the minimum and maximum values of $u_0(\x)$. 

A numerical  scheme should be maximum principle preserving (MPP), but the fully discrete scheme from the previous section is in general not MMP.  Let $\tilde u_h^{n+1}$ denote the solution produced by the fully discrete scheme at time $t=t_{n+1}$. We apply a correction to obtain a solution that satisfies a discrete maximum principle.  Our approach enforces the maximum principle locally on each macro-element by modifying the mean value on each macro-element. Recall the macro-element partition from Section~\ref{sec:macroel}. For each macro-element $M$, we define the corrected solution by 
\begin{equation}\label{eq:def:limitsolution:mean}
u_h^{n+1}|_{K_j} = \tilde u_h^{n+1}|_{K_j}+\delta_M,  \qquad j \in N_M.
\end{equation}
Here $\delta_M$ is a constant on $M$ chosen so that the mean value of the corrected solution on $M$ satisfies the maximum principle. 
$N_M$ denotes the index set of elements $K_j$ in the active mesh $\mt_h$ that belong to $M$. 

For a macro-element $M \in \mcM_h$, we denote its part in $\Omega$ by
\begin{equation}
\KM=M\cap\Omega.
\end{equation}
We denote by $\bar{u}_{h}$ the piecewise constant function of macro-element mean values, 
\begin{equation}\label{eq:def:meanvalues:M}
\bar{u}_{h}|_M =\frac{1}{|\KM|}\sum_{j\in N_M}\int_{K_j \cap \Omega} u_h dx. 
\end{equation}

Let $\Fmh\subset\Fh$ denote the set of macro-element interfaces, consisting of facets shared by two distinct macro-elements and the boundary $\Gamma$:
\begin{align}\label{eq:def:fmh}
\Fmh = \{e\cap\Omega:e=M_1\cap M_2, M_1\neq M_2,  M_1,M_2 \in \mcM_h \} \cup \{K\cap\Gamma: K\in\mcT_\Gamma\}.
\end{align}
The MPP flux limiting procedure corresponds to modifying the flux on each facet $\xi \in\Fmh$ to control the macro-element mean values. We now describe the flux limiting procedure on each macro-element. Assume that for a macro-element $M$ the facets in $\Fmh$ that have an intersection with M are denoted by $\FMh$. 

Taking $v_h=1$ on the macro-element $M$ and $v_h=0$ elsewhere in \eqref{eq:def:scheme:fullyscheme}, we obtain the following formula for the macro-element mean value:
\begin{align}\label{eq:scheme:mean:low}
\uhtildemn=\uhmo- \frac{\Delta t}{|\KM|}\sum_{\xi \in \FMh} \int_{\xi} \wf_{\xi}( u_h^n) ds.  
\end{align} 
Here $\wf_{\xi}$ is the numerical flux from \eqref{eq:numericalFlux} and \eqref{eq:def:LF}. 
Following the parameterized MPP limiter of \cite{xiong2015high}, we modify the numerical fluxes used in the macro-element mean-value update. On each facet $\xi \in\FMh$  (excluding interior facets of the macro-element on which ghost penalty stabilization is applied) we modify the flux and compute
 \begin{align}\label{eq:scheme:meanmacro}
\bar{u}_{h}^{n+1}|_M&=\uhmo- \frac{\Delta t}{|\KM|}\sum_{\xi \in \FMh} \int_{\xi} \liF(u_h^n) ds,\\
\liF(u_h) &= \theta_\xi \hiF(u_h) + (1-\theta_\xi) \loF(\bar u_h).\label{eq:def:limitflux}
\end{align}
In this paper we use a Lax-Friedrichs flux, see \eqref{eq:def:LF} where for the low order flux $\loF$ we use a globally determined value $\alpha_\xi$ of the maximum absolute Jacobian value, and for the high order flux $\hiF$ we use a local choice of $\alpha_\xi$. For the Euler method the difference between the high order flux and the low order flux is only in this parameter and in the argument to the flux, since the low order flux is applied on the macro-element mean values $\bar u_h^n$. 

 For each macro-element $M$ we thus set $\delta_M$ in \eqref{eq:def:limitsolution:mean} to
\begin{align}\label{eq:deltaM}
 \delta_M=\bar{u}_{h}^{n+1}|_M-\uhtildemn
 =\frac{\Delta t}{|\KM|}\sum_{\xi \in \FMh}\int_{\xi} (1-\theta_\xi) \left(\hiF(u_h^n)-\loF(\bar u_h^n) \right) ds.
\end{align}
In~\eqref{eq:def:limitflux}, the parameter $\theta_\xi \in [0,1]$ is constant on each interface $\xi \in\Fmh$ and chosen to ensure that the maximum principle holds for the macro-element mean value. For the lowest order polynomial spaces, piecewise constants, i.e. $r=0$, we choose $\theta_\xi=0$. For higher order polynomial spaces, the selection of the parameter $\theta_\xi$ is described in the next section. In contrast to the standard setting, this modification of the flux to $\liF(u_h)$ is done on interfaces $\xi\in\Fmh$ shared by macro-elements and on the boundary, rather than on all facets.

The extension to the third-order SSP-RK scheme \eqref{eq:def:sspRK3} is straightforward. Since the flux correction is applied only at the final RK stage, the only modification compared to the forward Euler method is that the high-order flux $\hiF(u_h^n)$ is replaced by the Runge--Kutta flux
\begin{align}
\wf^{rk}_{\xi}(u_h)=\frac{1}{6}\wf_{\xi}(u_h^n)+\frac{1}{6}\wf_{\xi}(u_h^{(1)})+\frac{4}{6}\wf_{\xi}(u_h^{(2)} ).\label{eq:def:highflux:RK}
\end{align}
Here $u_h^{(1)}$ and $u_h^{(2)}$ are defined by \eqref{eq:def:sspRK3}. The low order flux is exactly the same as before and uses the mean values of $u_h^n$ on $M$, i.e., $\loF(\bar u_h^n)$. After the third stage of the third order RK scheme,  we apply the flux limiting to correct the solution. The flux limiting procedure for the RK method is stated in Algorithm~\ref{algo:mpp}. 
\begin{algorithm}[!htb]
		\caption{Flux limiting}\label{algo:mpp}
		\textbf{Function} \text{Flux limiting}$(u_0, u_1, u_2,u_3, c_0, c_1, c_2, \tau)$ \\
\For {each macro-element $M$}
{
\For {each $\xi \in \FMh$}
{
Compute the high-order Runge--Kutta flux $\wf^{rk}_{\xi}=\sum_{j=1}^{3} c_j \wf_{\xi}(u_j)$  \\ 
Compute the low order numerical flux $\loF(\bar u_0)$ with $\bar u_0$ as in \eqref{eq:def:meanvalues:M}  \\
Compute $\theta_{\xi}^{M}$ following  \eqref{eq:def:Bm}-\eqref{eq:def:LambdaM}  \\
}
}
\For {each $\xi \in \Fmh$}{ Compute $\theta_\xi$ from \eqref{eq:def:thetai}. \\}
\For {each macro-element $M$}
{Compute $\delta_M$: 
$ \delta_M =\frac{\tau}{|\KM|}\sum_{\xi \in \FMh}\int_{\xi} (1-\theta_\xi) \left(\wf^{rk}_{\xi}-\loF(\bar u_0) \right) ds$. \\
\For { each element $j \in N_M$}
 {Compute the corrected solution as in  \eqref{eq:def:limitsolution:mean}: 
$u_h^{n+1}|_{K_j} =  u_3|_{K_j}+\delta_M$.\\
}
}
\Return $u_h^{n+1}$;		
	\end{algorithm}	
 
\newcommand{\Bm}{\mathbb{B}^\mm}
\newcommand{\BM}{\mathbb{B}^\MM}
\newcommand{\Cm}{\mathbb{C}^\mm}
\newcommand{\CM}{\mathbb{C}^\MM}

 \subsubsection{Parameters in the flux limiting}\label{sec:fluxlimit:para}
 In this subsection, we describe how to choose the parameters $\theta_\xi$, $\xi\in\Fmh$, such that the macro-element average defined by \eqref{eq:scheme:meanmacro} satisfies the maximum principle, i.e., on each macro-element $M \in \mcM_h$
\begin{align}\label{eq:def:fluxlimit:ieq}
\um \leq \uhmn=
\uhmo- \frac{\Delta t}{|\KM|}\sum_{\xi \in \FMh} \int_{\xi} \liF(u_h^n) ds
 \leq \uM.
\end{align}
In Section~\ref{sec:stability}, we prove in Theorem~3.5 that if the solution at time level $n$ is piecewise constant on the active mesh and its macro-element average, see equation \eqref{eq:def:meanvalues:M},  satisfies the maximum principle, then the solution at time level $n+1$ also satisfies the maximum principle. Furthermore, for solutions with polynomial degree $r\ge1$, setting $\theta_\xi=0$ in \eqref{eq:def:limitflux} yields a low-order update that satisfies the maximum principle, 
 that is for all macro-elements $M$ we have
\begin{align}\label{eq:def:Bm}
\Bm:=& \uhmo- \frac{\Delta t}{|\KM|}\sum_{\xi \in \FMh}\int_{\xi} \loF(\bar u_h^n)\, ds - \um \;\geq 0,\\
\BM:=&\uM -\uhmo+ \frac{\Delta t}{|\KM|}\sum_{\xi \in \FMh} \int_{\xi} \loF(\bar u_h^n) \, ds \;\geq 0.\label{eq:def:BM}
\end{align}
To retain as much of the high-order accuracy as possible, we therefore seek to choose $\theta_\xi$ as close to $1$ as permitted by the maximum-principle constraints.  
At each time level,  define
\begin{align}\label{eq:def:cm}
&\Cm := \frac{\Delta t}{|\KM|}\sum_{\xi \in \FMh}   \mathbb{F}(\beta_\xi),\quad
  \CM := \frac{\Delta t}{|\KM|}\sum_{\xi \in \FMh}  \mathbb{F}\!\left(-\beta_\xi \right), \text{ with }\\
&\beta_\xi=  \int_{\xi} (\wf^\mathcal{H}_{\xi}(u_h)-\loF(\bar u_h^n))\, ds, \quad 
\mathbb{F}(\beta_\xi) =
\begin{cases}
\beta_\xi, & \beta_\xi > 0, \\
0, & \text{otherwise}.
\end{cases}
\end{align}
with 
\begin{align}
\wf^\mathcal{H}_{\xi}(u_h)=
\begin{cases}
\hiF(u_h^n), & \textrm{Euler} , \\
 \frac{1}{6}\wf_{\xi}(u_h^n)+\frac{1}{6}\wf_{\xi}(u_h^{(1)})+\frac{4}{6}\wf_{\xi}(u_h^{(2)}), & \textrm{SSP-RK}.
\end{cases}
\end{align}

We first compute candidate limiter parameters for each facet $\xi$ with respect to the macro-element $M$ as 
\begin{align}\label{eq:def:lambdai}
\Lambda^\mm_{0,\xi} =
	 	\begin{cases}
	 		\min \left\{1,\dfrac{\Bm}{\Cm + \epsilon}\right\}, &  \beta_\xi \geq 0
			\\[1ex]
	 		1, & \text{else}
	 	\end{cases}
, \,
\Lambda^\MM_{0,\xi} =
	 	\begin{cases}
	 		\min \left\{1,\dfrac{\BM}{\CM + \epsilon}\right\}, & \beta_\xi \leq 0
			\\[1ex]
	 		1, & \text{else}.
	 	\end{cases} 
\end{align}
Here $\epsilon = 10^{-40}$ is a small constant  introduced to avoid division by zero. 
The parameters $\Lambda^\mm_{0,\xi}$ and $\Lambda^\MM_{0,\xi}$ are constructed to enforce the lower and upper bounds,  
$\uhmn \geq \um$ and $\uhmn \leq \uM$, respectively.  We then combine these bounds to obtain a single admissible local parameter for each facet of the macro-element:
\begin{align}\label{eq:def:LambdaM}
\theta_{\xi}^{M} = \min\!\left\{ \Lambda^\mm_{0,\xi}, \Lambda^\MM_{0,\xi} \right\}.
\end{align}
This procedure is repeated for all macro-elements. If $\xi\in\Fmh$ is part of a facet shared by two macro-elements, say $M$ and $M_i$, we take the neighbouring macro-element $M_i$ into account when we set the limiting parameter $\theta_\xi$. Thus, we set 
\begin{align}\label{eq:def:thetai}
	\theta_\xi = 
	\begin{cases}
	 		\min\!\left\{ \theta_{\xi}^{M}, \theta_{\xi}^{M_{i}}\right\}, & \textrm{$\xi$ is part of a facet shared by two macro-elements,} 
						\\[1ex]
	 		\theta_{\xi}^{M}, &  \textrm{ $\xi$ is part of the boundary $\Gamma$.} 
	 	\end{cases} 
\end{align}

\newcommand{\Muhm}{\bar{u}_{h}^{\MM}}  
\newcommand{\muhm}{\bar{u}_{h}^{\mm}}  
\newcommand{\Mqhm}{ q_{h}^{\MM}}  
\newcommand{\mqhm}{q_{h}^{\mm}}  

\subsection{Slope and bound-preserving limiting procedures}
For problems with discontinuous solutions, the MPP flux limiter alone is generally insufficient to suppress oscillations arising from the high-order approximation, as is also the case for standard DG methods, see \cite{10.1098/rspa.2011.0153, xiong2015high}. We also note that the MPP flux limiter only guarantees that mean values obey the maximum principle, both in a standard fitted DG setting and in our cut setting.  Additional limiting procedures are therefore required.  
Here we introduce versions of the Barth--Jespersen (BJ) \cite{barth1989design} slope limiter and the Zhang--Shu bound-preserving (BP) limiter, which can be used in the unfitted DG setting.  
 
\subsubsection{Slope limiter}\label{sec:slopelimiter}
To control oscillations near discontinuities, we adapt the Barth--Jespersen limiter \cite{barth1989design}.  Consider a macro-element $\M \in  \mcM_h$. After each Runge-Kutta stage, we modify the solution in each sub-element $K_j\subset \M,j \in N_M$ by employing the BJ limiter, to compute a new limited solution $u_h^{s}$ by
\begin{equation}\label{eq:slopelimit:uh}
	u_h^{s}\big|_{K_j} \;=\; \varpi_s\, u_h\big|_{K_j}+ (1-\varpi_s)\,\uhm,
		\qquad
	K_j\subset \M,
\end{equation}
and $\varpi_s$ is chosen as follows.

Let $S_{\M}$ denote the set consisting of the macro-element $\M$ and all neighboring macro-elements that share at least one vertex with $\M$. 
Let $\mathcal{Q}_{\M}$ denote the set of quadrature points associated with all sub-elements $K_j\subset \M$, $j \in N_M$. Define
\begin{equation}
\Muhm=\max_{\M \in S_{\M}} \bar u_h|_{\M}, 
\muhm=\min_{\M \in S_{\M}} \bar u_h|_{\M}, 
\Mqhm=\max_{x \in \mathcal{Q}_{\M}} u_h(x),
\mqhm=\min_{x \in \mathcal{Q}_{\M}} u_h(x).
\end{equation}
We determine the parameter \(\varpi_s\) as
\begin{equation}\label{eq:def:coefficient in BJ limiter}
\varpi_s=\min \left\{ \varpi_s^{\MM},\varpi_s^{\mm} \right\},
\end{equation}
where
\begin{equation*}
\begin{aligned}
\varpi_s^{\MM}&=
\begin{cases}
\dfrac{\Muhm-\uhm}{\Mqhm-\uhm+\varepsilon}, & \Mqhm>\Muhm, \\[1.2ex]
1,  & \Mqhm\le \Muhm,
\end{cases}
\,
\varpi_s^{\mm}=
\begin{cases}
\dfrac{\muhm-\uhm}{\mqhm-\uhm-\varepsilon}, & \mqhm<\muhm, \\[1.2ex]
1, & \mqhm\ge \muhm.
\end{cases}
\end{aligned}
\end{equation*}
Here $\varepsilon = 10^{-40}$ is added to avoid division by zero.  Applying \eqref{eq:slopelimit:uh} then yields the limited solution $u_h^{s}$. We note that the limiter is implemented as a Barth--Jespersen-type scaling procedure on each macro-element. The full polynomial deviation from the macro-element average is multiplied by a single factor, instead of applying the limiter only to the slope components.  
For strong shock problems, the limiter may alternatively be applied only to the slope components while setting the higher-order moments to zero.

\subsubsection{Bound-preserving limiter} \label{sec:boundpreserv}
As in standard DG methods for hyperbolic conservation laws on fitted meshes,  the MPP flux limiter guarantees only that the mean values are within the interval $[\um,\uM]$, in our case only on macro-elements.   However, a high-order polynomial solution within each element, as well as the corresponding element-wise mean values, may still violate these bounds, even if the macro-element averages satisfy the maximum principle.  To address this issue, we additionally apply the Zhang-Shu BP limiter \cite{Zhang2011MaximumPrincipleSatisfyingAP} at the final stage of each RK step. This guarantees that the numerical polynomial solution remains bounded and thereby enhances the stability. In our setting, the limiting parameter is computed on each macro-element rather than on each individual element.  

Consider a macro-element $M\in\mcM_h$. At the final stage of each RK step, for macro-element $M$ we define the limited solution
\begin{align}\label{eq:def:bplimit:solution}
{u}^b_h|_{K_j}=\varpi_{{b}}{u}_h|_{K_j}+(1-\varpi_{{b}})\uhm,	 \qquad K_j\subset \M ,
\end{align}
where $\varpi_b$ is constant on each macro-element and computed as
\begin{align}\label{eq:def:varpi_b}
		\varpi_b
		= \min\left\{
		1,\;
		\frac{\uM - \uhm}{\Mqhm - \uhm},\;
		\frac{\uhm - \um}{\uhm - \mqhm}
		\right\}.
\end{align}

\subsection{Fully discrete maximum-principle-preserving scheme}
The complete fully discrete algorithm, based on the SSP-RK time discretization and flux limiting is summarized in Algorithm \ref{algoim:cutdg scheme}. At each time step, the solution is advanced using the third-order SSP-RK method \eqref{eq:def:sspRK3}. For problems containing discontinuities, the slope limiter from Section~\ref{sec:slopelimiter} may be applied after each stage of the first two stages. After the third stage, we first apply the flux-limiting procedure described in Sections~\ref{sec:fluxlim} to ensure that the macro-element mean values satisfy the maximum principle. Subsequently, the slope limiter is again applied in the presence of discontinuities. Finally, to enforce the maximum principle for the polynomial solution, the bound-preserving limiter from Section~\ref{sec:boundpreserv} is applied. The resulting numerical solution satisfies the maximum principle and serves as the approximation at the new time level. 
\begin{algorithm}[!htb]
\caption{Fully discrete MPP CutDG schem}\label{algoim:cutdg scheme}
\KwData{Given the initial condition $u_0$}
\KwResult{Output the MPP approximation at the final time $T$}
Set the initial time $t = 0$ and $n = 0$, and initialize $u_h^0=u_h(\cdot,0)$ by \eqref{eq:def:initial};\\
\While{$t \leq T$}{
			Compute the time-step size $\Delta t$;  
			\tcp{Third-order SSP Runge--Kutta discretization}
			$u_h^{(1)} = \text{RK Stage}(u_h^n, 0, 0, 1, 0, 0, \Delta t)$ \tcp*{first stage}
			If (slope limiting): $u_h^{(1)}$= BJ limiter ($u_h^{(1)}$)
			
			$u_h^{(2)} = \text{RK Stage}(u_h^n, u_h^{(1)}, 0, 1, 1, 0, \Delta t/4)$ \tcp*{second stage}
			If (slope limiting): $u_h^{(2)}$= BJ limiter ($u_h^{(2)}$)
			
			$u_h^{(3)} = \text{RK Stage}(u_h^n, u_h^{(1)}, u_h^{(2)}, 1, 1, 4, \Delta t/6)$ \tcp*{third stage }			
			$u_h^{(3)} = \text{flux limiting}(u_h^n, u_h^{(1)}, u_h^{(2)}, u_h^{(3)}, 1, 1, 4, \Delta t/6)$\tcp*[l]{flux limiter}
			
			If (slope limiting): $u_h^{(3)}$= BJ limiter ($u_h^{(3)}$)\\
			If (bound-preserving limiter) $u_h^{(3)} = \text{Bound-preserving limiter}(u_h^{(3)}$) \tcp{Limit the polynomial solution to prescribed bounds.}
			
			Set $u_h^{n+1}=u_h^{(3)}$;
			$n \leftarrow n+1$;
			$t \leftarrow t+\Delta t$;
		}
\end{algorithm}

\section {Main analytical results}
In this section, we establish the $L^2$-stability of the semi-discrete scheme and prove the MPP property of the fully discrete scheme. We also derive an a priori error estimate for the linear advection equation.

\subsection{$L^2$-stability of the semi-discrete scheme}\label{sec:stability}
In this subsection, we establish $L^2$-stability of the semi-discrete scheme using entropy arguments, following the analysis of DG methods in  \cite{jiang1994cell, shu2009discontinuous}. Throughout the analysis, we consider the quadratic entropy function $U(u) = \frac{1}{2}u^2$.

\begin{theorem}
Consider the semi-discrete scheme~\eqref{eq:def:scheme}--\eqref{eq:def:initial} for the scalar conservation law \eqref{eq:model2D}.  
Assume that $\F\in C^1(\mathbb R)^d$, numerical flux $\wf$ is the generalized Lax-Friedrichs flux \eqref{eq:def:LF} and  $\gm \ge 0$, $\ga \ge 0$. Then, in the case of periodic boundary conditions, the semi-discrete solution $u_h$ satisfies the  $L^2$-stability estimate
\begin{equation}\label{eq:def:stab:period}
(u_h(\cdot,T),u_h(\cdot,T))_s\leq  (u_h(\cdot,0),u_h(\cdot,o))_s. 
\end{equation}  
For inflow-outflow boundary conditions, we assume the generalized Lax-Friedrichs parameter satisfies $\diffc>1$ on $\xi\subset\bc^-$. Then there  exists a constant $C>0$, independent of $u_h$, $g$, and the cut configuration, such that  
\begin{equation}\label{eq:def:stab:inflow}
(u_h(\cdot,T),u_h(\cdot,T))_s\leq  (u_h(\cdot,0),u_h(\cdot,o))_s+
C\int_0^T \|g\|_{\bc^-}^2\,dt.  
\end{equation}  
\end{theorem}

\begin{proof}
Choosing $v_h=u_h$ in \eqref{eq:def:scheme} yields
\begin{align}
\frac{1}{2}\frac{d}{dt}\int_\Omega u_h^2 d\x+\gm \stab_{h,1}(u_{h,t},u_h)+a_h(u_h,u_h)+\ga \stab_{h,0}(u_h,u_h)=0.
\end{align}
For the quadratic entropy $U(u)=u^2/2$, let $\G$ be a primitive of
the physical flux $\F$, i.e., $\G'(u) = \F(u)$. Hence, for any $w\in\V$, 
$
\nabla\cdot\G(w)
=
\G'(w)\cdot\nabla w
=
\F(w)\cdot\nabla w.
$  
Using the definition of $a_h(\cdot,\cdot)$ and integrating by parts elementwise, we obtain
\begin{align}\label{eq:stability:eq1}
&a_h(w,w)=
-\sum_{K\in\mt_h} (1, \nabla \cdot \G(w))_{K\cap\Omega}
-\sum_{\xi \in \mathcal{F}_h}
\int_\xi \wf_\xi(w)\, [w]_\xi\,ds 
 =\\ 
\sum_{\xi\in\mathcal F_h^e}&
\int_\xi
\Big(
[\G(w)]_\xi\cdot\n_\xi
-\widehat F_\xi(w)[w]_\xi
\Big)\,ds
+\sum_{\xi\in\mathcal F_h^\bc}
\int_\xi
\Big( \widehat F_\xi(w)w^- - \G(w^-)\cdot\n_\xi \Big)\,ds.\notag
\end{align}

We first consider the contribution from the interior interfaces $\xi\in\mathcal F_h^e$. 
By the mean value theorem, there exists 
$\widetilde w \in [\min(w^-,w^+),\max(w^-,w^+)]$, 
such that
$[\G(w)]_\xi\cdot\n_\xi
=
\F(\widetilde w)\cdot\n_\xi\,[w]_\xi.$ 
Since $\widetilde w$ lies between $w^-$ and $w^+$ and the numerical flux satisfies the monotonicity property (F.2), it follows that 
\begin{equation} \label{eq:boundGmF}
[\G(w)]_\xi\cdot\n_\xi
-
\widehat F_\xi(w)[w]_\xi
=
\bigl(
\F(\widetilde w)\cdot\n_\xi
-
\widehat F_\xi(w) 
\bigr)
[w]_\xi 
\ge 0.
\end{equation}

For $\xi$ on the boundary, we consider different cases.  For periodic boundary conditions, boundary interfaces can be treated in the same way as interior interfaces. Consequently 
$
a_h(u_h,u_h)\ge 0,
$
and therefore 
\begin{align}
\frac{1}{2}\frac{d}{dt}\int_\Omega u_h^2 \,d\x+\gm \stab_{h,1}(u_{h,t},u_h)+\ga\stab_{h,0}(u_h,u_h)\leq 0.
\end{align}
Integrating over the time interval $[0,T]$, using the definition \eqref{stable:2D} of the stabilization form together with the positivity of  $\gm$  and $\ga$ yields the stability estimate \eqref{eq:def:stab:period} for the periodic boundary condition.

On an outflow boundary interface $\xi\in\bc^+$, we have $ w^+=w^-.$ 
By consistency (F.1) 
$
\widehat F_\xi(w)
=
\F(w)\cdot\n_\xi.
$
Hence,
$\widehat F_\xi(w)w^-
-
\G(w^-)\cdot\n_\xi
=
\bigl(
w\F(w)-\G(w)
\bigr)\cdot\n_\xi= \Eflux(w)\cdot\n_\xi.
$ 
  Since $\xi\subset\bc^+$, we have
$
\F'(u)\cdot\n_\xi>0,
$ 
so the scalar normal flux $\F(u)\cdot\n_\xi$ is monotone increasing. Using the identity
$
\Eflux(w)\cdot\n_\xi
=
\int_0^w
\bigl(
\F(w)-\F(s)
\bigr)\cdot\n_\xi\,ds,
$ 
it follows that
$
\Eflux(w)\cdot\n_\xi\ge0. 
$
Hence, the outflow boundary contribution is nonnegative.

For an inflow boundary face $\xi\in\mathcal F_h^{\bc}$, the exterior state is
given by the prescribed boundary data  $w^+=g.$ 
Using the identity $ w^-=g+(w^--g)$, 
we write the boundary contribution as
\begin{align}
&\widehat F_\xi(w)w^-
-\G(w^-)\cdot \n_\xi
=
\Big(
\wf\big(w^-, g;\, \n_\xi \big)(w^- - g)
-
(\G(w^-)-\G(g))\cdot \n_\xi
\Big)
\notag\\
&\qquad
+
\Big(
\wf\big(w^-, g;\, \n_\xi \big)g
-\G(g)\cdot\n_\xi
\Big).
\label{eq:boundary-decomposition}
\end{align}
Let 
$
D_\xi(w^-,g)
:=
\widehat F(w^-,g;\n_\xi)(w^- - g)
-
(\G(w^-)-\G(g))\cdot\n_\xi .
$ 
For the generalized Lax--Friedrichs boundary flux \eqref{eq:boundary-gllf-flux}, 
\begin{align}
D_\xi(w^-,g)
&=
\frac{\lambda_\xi}{2}(w^- - g)^2
+
\int_g^{w^-}
\left(
\frac{\F(w^-)+\F(g)}{2}
-\F(s)
\right)\cdot\n_\xi\,ds .
\end{align}
Since
\[
\left|
\int_g^{w^-}
\left(
\frac{\F(w^-)+\F(g)}{2}
-\F(s)
\right)\cdot\n_\xi\,ds
\right|
\le
\frac{\alpha_\xi}{2}(w^- - g)^2,
\]
and since $\lambda_\xi=\diffc \alpha_\xi$ with $\diffc>1$, we obtain
\begin{equation}
D_\xi(w^-,g)\ge \frac{\diffc-1}{2}\alpha_\xi(w^- - g)^2. 
\label{eq:D-nonnegative}
\end{equation}
Using consistency 
$
\widehat F(g,g;\n_\xi)
=
\F(g)\cdot\n_\xi,
$
we obtain
\begin{align}
\widehat F(w^-,g;\n_\xi)g
-\G(g)\cdot\n_\xi
=
\Eflux(g)\cdot\n_\xi
+
g\Bigl(
\widehat F(w^-,g;\n_\xi)
-
\widehat F(g,g;\n_\xi)
\Bigr).
\label{eq:boundary-data-decomposition}
\end{align}
For the second term, using the definition of the flux \eqref{eq:def:LF} have 
\begin{align}
\left| \widehat F(w^-,g;\n_\xi) - \widehat F(g,g;\n_\xi) \right|
\le \frac{\diffc+1}{2}  \alpha_\xi  |w^- - g|.
\end{align}
Hence, by Young's inequality, 
\begin{align}\label{eq:mixed-boundary-term}
\left|
g\Bigl(
\widehat F(w^-,g;\n_\xi)
-
\widehat F(g,g;\n_\xi)
\Bigr)
\right|
&\le
\frac{\diffc-1}{4} \alpha_\xi |w^- - g|^2  
+ \frac{(\diffc+1)^2}{4(\diffc-1)} \alpha_\xi |g|^2. 
\end{align} 
Using
$
\Eflux(g)\cdot\n_\xi
=
\int_0^g
(\F(g)-\F(s))\cdot\n_\xi\,ds,
$
and $ |(\F(g)-\F(s))\cdot\n_\xi| \le \alpha_\xi |g-s|$ gives 
\[
|\Eflux(g)\cdot\n_\xi| =
\left| \int_0^g (\F(g)-\F(s))\cdot\n_\xi\,ds \right| 
\le
\frac12\alpha_\xi |g|^2.
\]
Therefore
\begin{equation}
\widehat F(w^-,g;\n_\xi)g
-\G(g)\cdot\n_\xi
\ge
- \frac12 {\alpha_\xi}  |g|^2
- \frac{\diffc-1}{4} \alpha_\xi |w^- - g|^2  
- \frac{(\diffc+1)^2}{4(\diffc-1)}  \alpha_\xi |g|^2.
\label{eq:boundary-data-estimate}
\end{equation}
Combining \eqref{eq:boundary-decomposition}, \eqref{eq:D-nonnegative}, and \eqref{eq:boundary-data-estimate} yields the following bound 
\begin{align*}
\widehat F_\xi(w)w^-
-\G(w^-)\cdot\n_\xi
\ge
\frac{\diffc-1}{4} \alpha_\xi|w^- - g|^2
- \frac{(\diffc+1)^2}{4(\diffc-1)}  \alpha_\xi |g|^2
-\frac12 \alpha_\xi |g|^2. 
\end{align*}
Summing over inflow facets gives 
\begin{align*}
a_h(w,w)
\ge
\sum_{\xi\subset\Gamma^-}
\frac{\diffc-1}{4}\int_\xi  \alpha_\xi | w^- - g|^2
\,ds
-
\sum_{\xi\subset\Gamma^-}
\int_\xi \Bigl( \frac{(\diffc+1)^2}{4(\diffc-1)}+\frac12\Bigr)  \alpha_\xi  |g|^2 ds. 
\end{align*}

Therefore, inserting the above estimate into the energy identity yields 
\begin{align*} 
\frac12\frac{d}{dt}\|u_h\|_\Omega^2 
&+ \gm \stab_{h,1}((u_h)_t,u_h) + \ga \stab_{h,0}(u_h,u_h) 
 + \sum_{\xi\subset\Gamma^-} \frac{\diffc-1}{4} \int_\xi \alpha_\xi |u_h^- - g|^2\,ds \\ 
 &\le C\sum_{\xi\subset\Gamma^-} \int_\xi |g|^2\,ds . 
 \end{align*} 
Dropping the nonnegative boundary term (since $\diffc>1$) and $\ga \stab_{h,0}(u_h,u_h)$ and integrating in time yields the estimate \eqref{eq:def:stab:inflow}. 

In summary, the semi-discrete CutDG scheme is $L^2$-stable. The stabilized energy is therefore bounded by the corresponding initial energy and, in the presence of inflow boundaries,  by the prescribed boundary data. Moreover, the estimate is independent of the cut configuration and therefore does not deteriorate in the presence of arbitrarily small cut cells.
\end{proof}

\begin{remark}
If the scalar normal flux $\F(u)\cdot\n_\bc$ is monotonically decreasing on the interval $[\min(w^-,g),\max(w^-,g)]$ at every
point of $\Gamma^-$, which is satisfied, for example, for linear advection, one can use the simpler pure upwind boundary flux
\begin{equation}
\widehat{\boldsymbol F}(w^-,g,\n_\xi) = \boldsymbol F(g)
\label{eq:inflow_upwind_flux}
\end{equation}
 on $\Gamma^-$ and the $L^2$ stability can also be proven with this flux. We have repeated all numerical experiments  in Section~\ref{sec:num} using the pure upwind flux on the inflow boundary and observed the same convergence rates and qualitatively identical solutions. 
\end{remark}

\begin{remark}
The nonlinear stability analysis relies on the assumption that the integrals appearing in \eqref{eq:stability:eq1} are evaluated exactly when performing the integration by parts. This assumption is standard in the analysis of DG methods on fitted meshes.
To avoid the assumption of exact integration in the stability analysis, one can use a flux splitting weak formulation alternatively. For example in case of Burgers' equation one can replace the bilinear formulation $a_h(u_h,v_h)$ in \eqref{scheme:2D:ah} by 
\begin{align}
a(u_h,v_h)= &-\frac{2}{3}\sum_{K\in\mt_h} (\F(u_h), \nabla v_h)_{K\cap\Omega}
+\frac{1}{3}\sum_{K\in\mt_h}(u_h\nabla u_h, v_h)_{K\cap\Omega} \notag\\
&-\sum_{\xi \in \mathcal{F}_h}\int_\xi ( \wf_\xi(u_h)-\frac{1}{3}\F(u_h^{-})\cdot \n_{\xi} )\, [v_h]_\xi\,ds.
\label{scheme:2D:ah:split}
\end{align}
In subsection \ref{sec:num:nonlinear}, the results from the scheme \eqref{eq:def:scheme} with \eqref{scheme:2D:ah:split} and MPP flux limiting are shown in Figure \ref{fig:burgers result for smooth with flux splitting} and in Figure \ref{fig:burgersd_result_barth_split}. 
The numerical results show that the proposed method with the flux-splitting formulation \eqref{scheme:2D:ah:split} produces solutions that are nearly indistinguishable from those obtained using the original formulation \eqref{scheme:2D:ah}.
\end{remark}

We assume $u\in H^{r+1}(\Omega)$ is the exact solution of the hyperbolic conservation laws \eqref{eq:model2D} with  boundary condition \eqref{eq:def:boundrycon}. Using the stability estimate established above, the error analysis follows along the same lines as in \cite{fu2024bound}. 
\begin{theorem}\label{thm:scalarp:error}
Let $u(\x,t)$ denote the exact solution of \eqref{eq:model2D}--\eqref{eq:initialcond2D},  where $\F(u) = \be u$ and  $\nabla \cdot \boldsymbol{\beta} = 0$. Assume that  $u,\, u_t \in L^{\infty} \left([0,T]; H^{r+1}(\Omega)\right)$. Let $u_h \in\V$ be the numerical solution obtained using the CutDG scheme \eqref{eq:def:scheme} with parameters $\gm>0,\ga>0$.  Then the following a priori error estimate holds:
\begin{align}\label{eq:error:estimate}
\|u(\cdot,t) - u_h(\cdot,t)\|_{\Omega}^2 \leq C h^{2r}, \qquad t \in [0,T],
\end{align}
where $C$ is a constant independent of the mesh parameter $h$ and of how the interface intersects the mesh.
\end{theorem}


\subsection{The fully discretized scheme}\label{sec:fluxlimit}
In this subsections, we first establish a maximum-principle-preserving property for the low-order piecewise constant scheme. We then show that the proposed flux limiting strategy ensures that the macro-element mean values of the high-order scheme also satisfy the same maximum principle.

\subsubsection{Low order scheme}
In this subsection, we consider the low-order scheme obtained from the flux-limiting strategy \eqref{eq:def:limitflux}, setting $\theta_\xi=0$. We show that the resulting macro-element mean values \eqref{eq:scheme:mean:fluxlimiter} satisfy a discrete maximum principle. 

Recall that flux-limiting with piecewise constant polynomial approximation spaces corresponds to \eqref{eq:scheme:meanmacro} with the modified flux defined by \eqref{eq:def:limitflux} with $\theta_\xi=0$, that is 
\begin{align}\label{eq:scheme:mean:fluxlimiter}
\uhml&=\uhmo- \frac{\Delta t}{|\KM|}\sum_{\xi\in\FMh}\int_{\xi} \loF(\bar u_h) ds. 
\end{align}
Here $\KM=M\cap \Omega$ denotes the macro-element intersection, and $|\partial\KM|$ denotes the total length of its
boundary. By construction of the macro-elements, there exists a constant $C_\delta>0$,  depending on the aggregation parameter $\delta$
and the shape regularity of the background mesh, such that 
\begin{equation}
\frac{|\partial\KM|}{|\KM|}
\le C_\delta h^{-1}.
\end{equation}
Since $\delta$ is fixed independently of the mesh size $h$, the constant $C_\delta$ is independent of the cut configuration.

Assume that the piecewise constant approximation satisfies $\uhmo\in[\um,\uM]$ on every macro-element $M$. The following theorem shows that the low-order scheme \eqref{eq:scheme:mean:fluxlimiter} preserves these bounds under a CFL condition that scales as in the fitted case and therefore does not suffer from the severe time-step restriction associated with arbitrarily small cut cells.

\begin{theorem}\label{p0-bound preserving}
Consider the piecewise constant scheme \eqref{eq:scheme:mean:fluxlimiter} for the conservation law \eqref{eq:model2D}. 
Assume that the macro-element mean values, defined by~\eqref{eq:def:meanvalues:M},  satisfy $\uhmo \in [\um, \uM]$ for all macro-elements $M$.  Furthermore, assume that the time step $\Delta t$ satisfies 
\begin{align}\label{eq:limitscheme:timestep}
	\max\limits_{\KM\in \Omega}\frac{|\partial \KM|} {|\KM|} \lambda \Delta t \leq 1.
\end{align}
Then the updated macro-element mean values  satisfy 
$
\uhml \in [\um, \uM].
$ 
for all macro-elements.
\end{theorem}	
\begin{proof}
Let $L$ denote the total number of facets in $\FMh$. These include both facets shared with neighbouring macro-elements and boundary facets. 
Following the standard DG analysis in \cite{Zhang2011MaximumPrincipleSatisfyingAP}, we define
\begin{align}\label{eq:bp:proof:H}
H&(\bar{u}_0, \bar{u}_1, \cdots, \bar{u}_L; \KM) = \bar{u}_0- \frac{\Delta t}{|\KM|}\sum_{\xi\in\FMh}\int_{\xi} \wf_{\xi}\left( \bar{u}_0, \bar{u}_{i};\n_{\xi} \right)\,ds\\
&= \bar{u}_0-\frac{\Delta t}{2|\KM|}\sum_{\xi\in\FMh}\int_{\xi}\left(\F( \bar{u}_0) \cdot \n_{\xi}+\lambda \bar{u}_0\right) 
-\frac{\Delta t}{2|\KM|}\sum_{\xi\in\FMh}\int_{\xi}\left(\F( \bar{u}_i) \cdot \n_{\xi}-\lambda \bar{u}_i\right), \notag
\end{align}
where $\bar u_i$ denotes the value associated with the facet $\xi_i\in\FMh$.

Since
$
\left|
\F'(\bar u_i)\cdot\n_\xi
\right|
\le \lambda, i=0,\ldots,L,
$ 
and
$
\sum_{\xi\in\FMh} |\xi|
=
|\partial\KM|,
$ 
the time-step restriction \eqref{eq:limitscheme:timestep} implies
\begin{align}
\frac{\partial H}{\partial \bar{u}_0}&=1-\frac{\Delta t}{2|\KM|}\sum_{\xi\in\FMh}\int_{\xi}
\left( \F'(\bar u_0)\cdot\n_\xi +\lambda \right) \geq 0,\\
\frac{\partial H}{\partial \bar{u}_i}&=-\frac{\Delta t}{2|\KM|} 
\int_{\xi_i} \left( \F'(\bar u_i) \cdot \n_{\xi} -\lambda\right)\geq 0, \quad i=1,...,L,
\end{align}
where $\xi_i\in \FMh$ is the facet associated with $u_i$. Hence, $H$ is nondecreasing in each of its arguments.

For facets $\xi_i \in \FMh$ shared by macro-elements $M$ and $M_i$ we let  $\uhin$ denote the mean value on the macro-element $M_i$.  
For boundary facets $\xi\subset\Gamma$, the value $\uhin$ is replaced by the boundary state prescribed by the boundary condition. With the definition of $H$ in \eqref{eq:bp:proof:H}, we can rewrite the scheme \eqref{eq:scheme:mean:fluxlimiter}  as 
\begin{equation}\label{eq:def:meanvalues:M:2}
\uhml=H(\uhmo, \uhfn, \cdots, \uhln;\KM).
\end{equation}

Since by assumption $\uhmo,\uhin$, for $ i= 1,\cdots,L$ are all in the interval $[\um, \uM]$ and $H$ is nondecreasing in all of its arguments, we obtain
\[
H(\um,\ldots,\um;\KM)
\le
H(\uhmo,\uhfn,\ldots,\uhln;\KM)
\le
H(\uM,\ldots,\uM;\KM).
\]
By consistency of the numerical flux,
\[
H(\um,\ldots,\um;\KM)=\um,
\qquad
H(\uM,\ldots,\uM;\KM)=\uM.
\]
Therefore,
$
\um\le\uhml\le\uM.
$ 
Hence, the low-order scheme preserves the maximum principle for the macro-element mean values.
\end{proof}

\subsubsection{High order scheme}
We now show that, after applying the proposed flux limiter, the high-order scheme satisfies the maximum principle for the macro-element mean values. 

\begin{theorem}
Suppose the assumptions of Theorem~\ref{p0-bound preserving} are satisfied.  
Let the limiter parameters $\theta_\xi$ be chosen according to \eqref{eq:def:thetai}. 
Then the updated macro-element mean values produced by \eqref{eq:scheme:meanmacro} with the limited flux \eqref{eq:def:limitflux} satisfy
\begin{align}
\uhmn\in[\um,\uM].
\end{align}
\end{theorem}

\begin{proof}
 From the definitions of $\mathbb B^\mm$, $\mathbb B^\MM$ in
\eqref{eq:def:Bm}--\eqref{eq:def:BM}, 
$\mathbb C^\mm$, $\mathbb C^\MM$  in
\eqref{eq:def:cm}, and  
$\Lambda_{0,\xi}^{\mm}$, $\Lambda_{0,\xi}^{\MM}$
in  \eqref{eq:def:lambdai}, we have 
\begin{align*}
\uhmo&- \frac{\Delta t}{|\KM|} \sum_{\xi\in\FMh}\int_{\xi} \Bigl(\Lambda^\mm_{0,\xi} \hiF(u_h^n) + (1-\Lambda^\mm_{0,\xi}) \loF(\bar{u}_h^n)\Bigr) ds\\
&\geq  \mathbb{B}^\mm + \um - \Lambda^\mm_{0,\xi} \mathbb{C}^\mm\geq \um, \\
\uhmo &- \frac{\Delta t}{|\KM|} \sum_{\xi\in\FMh}\int_{\xi} \Bigl(\Lambda^\MM_{0,\xi} \hiF(u_h^n) + (1-\Lambda^\MM_{0,\xi}) \loF(\bar{u}_h^n)\Bigr) ds\\
&\leq 
\mathbb{B}^\MM + \uM - \Lambda^\MM_{0,\xi} \mathbb{C}^\MM \leq \uM .
\end{align*}
Hence, the admissibility condition \eqref{eq:def:fluxlimit:ieq} holds on the macro-element $M$ for all $\theta_\xi \in [0,\theta_\xi^M]$,  where $\theta_{\xi}^{M} $ is defined by \eqref{eq:def:LambdaM}. 
If a facet is shared by two macro-elements $M$ and $M_i$, the parameter $\theta_\xi$ is chosen according to \eqref{eq:def:thetai}. Consequently, 
$\theta_\xi\le \theta_\xi^M,$ 
$\theta_\xi\le\theta_\xi^{M_i},$
so that the admissibility condition is satisfied for both macro-elements simultaneously. Therefore, \eqref{eq:def:fluxlimit:ieq} holds on every macro-element, and the updated mean values satisfy
\[
\um \le \uhmn \le \uM.
\]
\end{proof}

\newcommand{\jaco}{\alpha}

\section{Numerical examples}\label{sec:num}
In this section, we present numerical examples illustrating the accuracy, stability, and maximum-principle-preserving properties of the proposed schemes.  The physical domain is embedded in a fixed uniformly refined Cartesian background mesh. In all computations, $\delta=0.2$ is used in \eqref{eq:largeel} to identify small cells. For the time discretization,  the third order SSP-RK method \eqref{eq:def:sspRK3} is used for the proposed method with  $P^1$ and $P^2$ approximations, and the fourth order SSP-RK method \cite{spiteri2002new} is used for the $P^3$ approximations. The time step is chosen as $\Delta t = 0.15\frac{h}{\jaco}, 0.08\frac{h}{\jaco}$, and $0.05\frac{h}{\jaco}$ for $P^1,P^2$, and $P^3$ approximations, respectively. Here $\jaco$ denotes the largest absolute value of the Jacobian  ${\partial \F(u)\cdot\n}/{\partial u}$ over all the interfaces $\xi\in \mathcal{F}_h$.  Dirichlet data are imposed on inflow boundaries, while extrapolation is used at outflow boundaries. In the following subsections, the accuracy and MPP property of the proposed method are demonstrated by solving both smooth and discontinuous problems.

\subsection{Linear problems with constant coefficients}
We first consider the linear advection problem with $ \F(u)=\be u$ in \eqref{eq:model2D} with constant $\be$
\begin{align}\label{eq:linear}
    \partial_{t} u+ \nabla \cdot (\be u)=0.
\end{align}

\subsubsection{A two dimensional problem}\label{test: flower}
We take $\be=(2,2)$ and the initial condition as $u_0(x,y)=0.5+\sin(0.5 \pi (x+y))$. 
The physical domain is defined as in \cite{gurkan2020stabilized}, where domain's boundary $\Gamma$ is  implicitly described by the level set function
\begin{equation}
    \Gamma = \{(x,y) \in \mathbb{R}^2 \mid \phi(x,y) = 0 \}, \,
    \phi(x,y) = \sqrt{x^2 + y^2} - r_0 - r_1 \cos\bigl(5\,\text{atan}_2(y,x)\bigr),
\end{equation}
where $r_0 = 0.5$ and $r_1 = 0.15$. We set the physical domain as $\Omega=\{ (x,y) \in \Omega: \phi(x,y) \leq 0\}$. We compute this problem up to $T=0.5$.
A coarse mesh together with our numerical results is shown in Figure~\ref{fig: linear 2D result}.  As expected, when the MPP flux limiting is applied, no overshoots or undershoots of the mean values on macro-elements are observed. The numerical errors and convergence rates shown in the right panel of Figure~\ref{fig: linear 2D result},  indicate optimal convergence rates. 
\begin{figure}[!htbp]
    \begin{center}
        \includegraphics[width=0.235\textwidth]{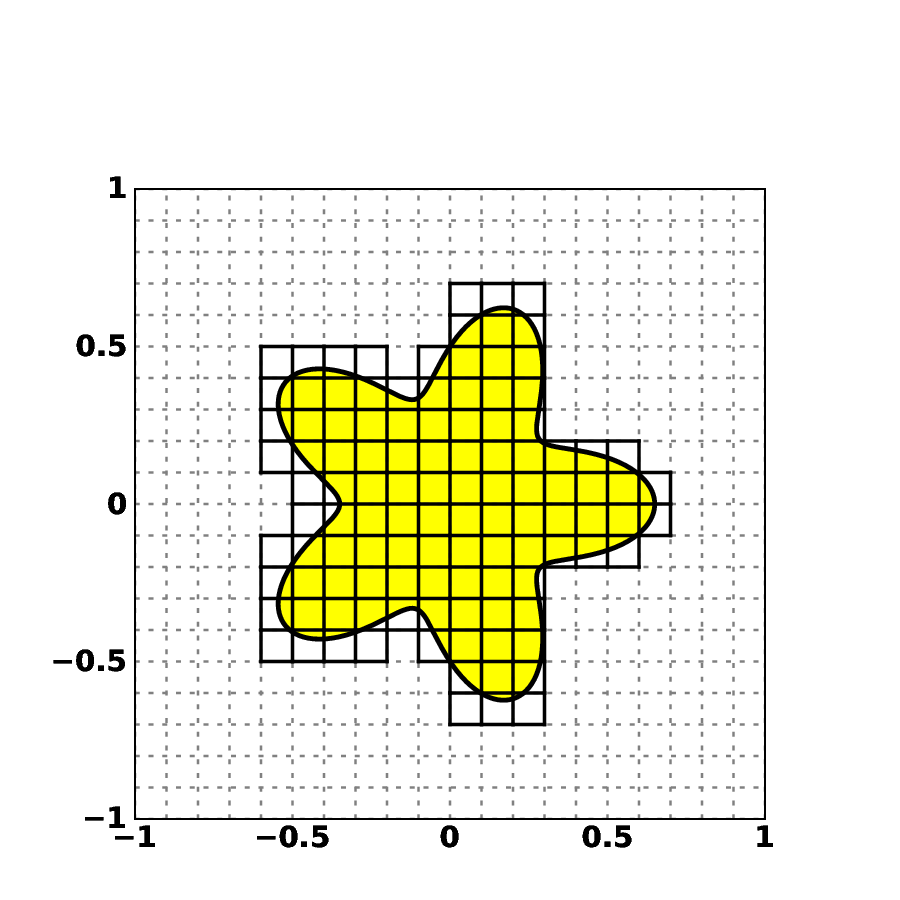}
        \includegraphics[width=0.235\textwidth]{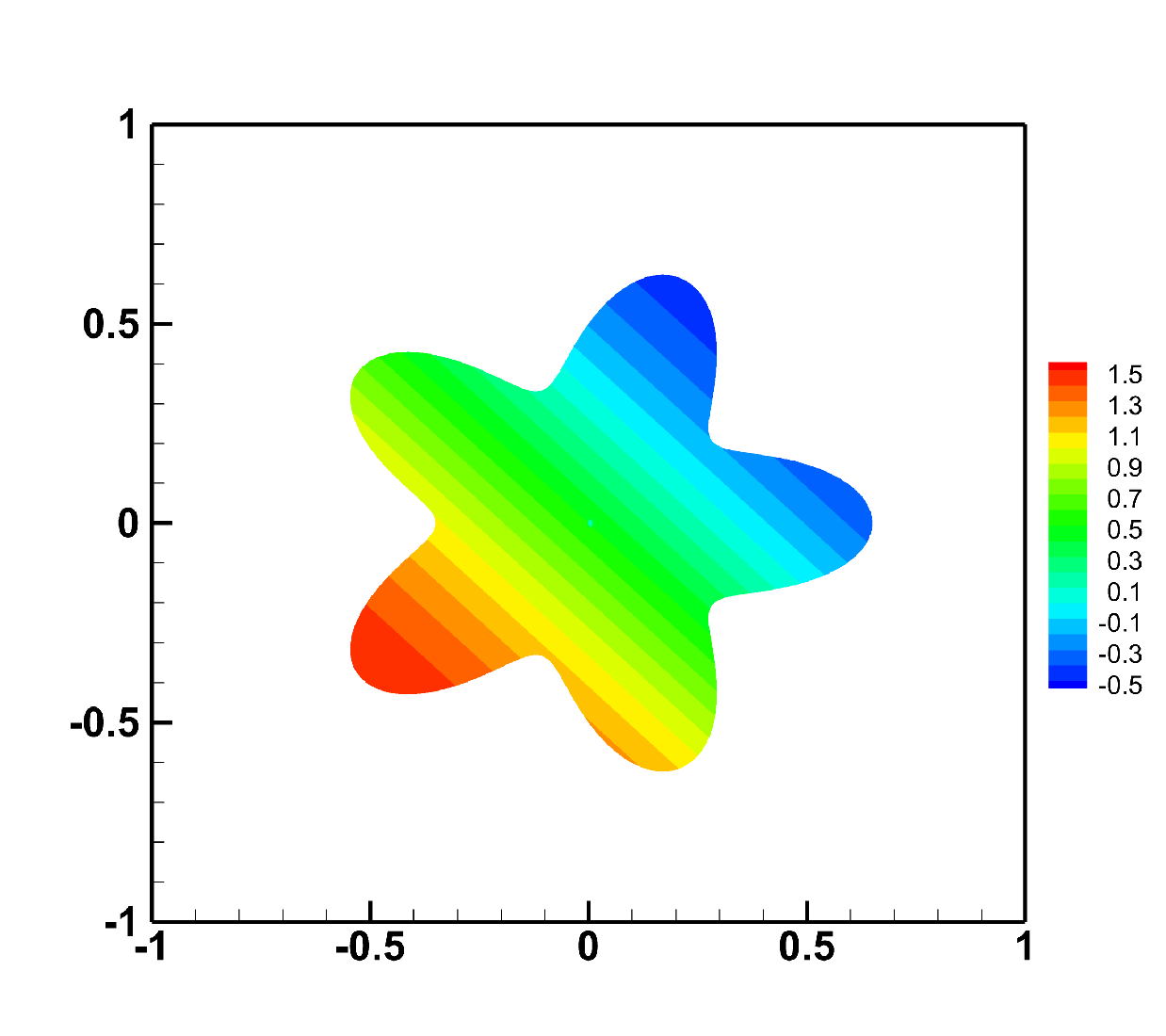}
        \includegraphics[width=0.25\textwidth]{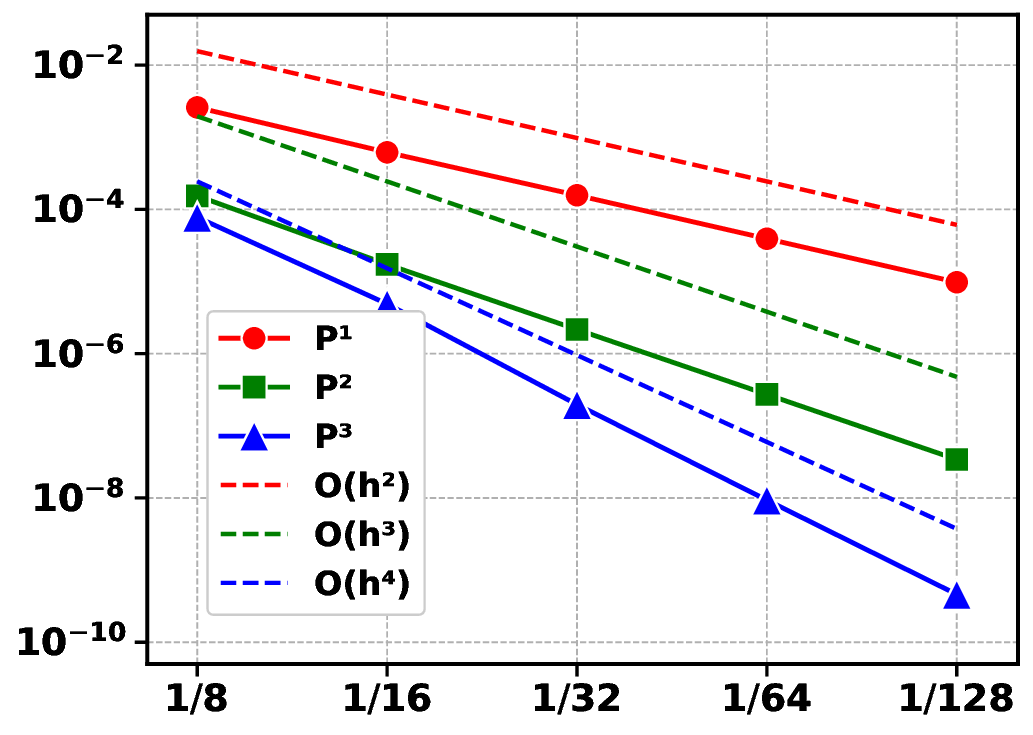}
        \includegraphics[width=0.25\textwidth]{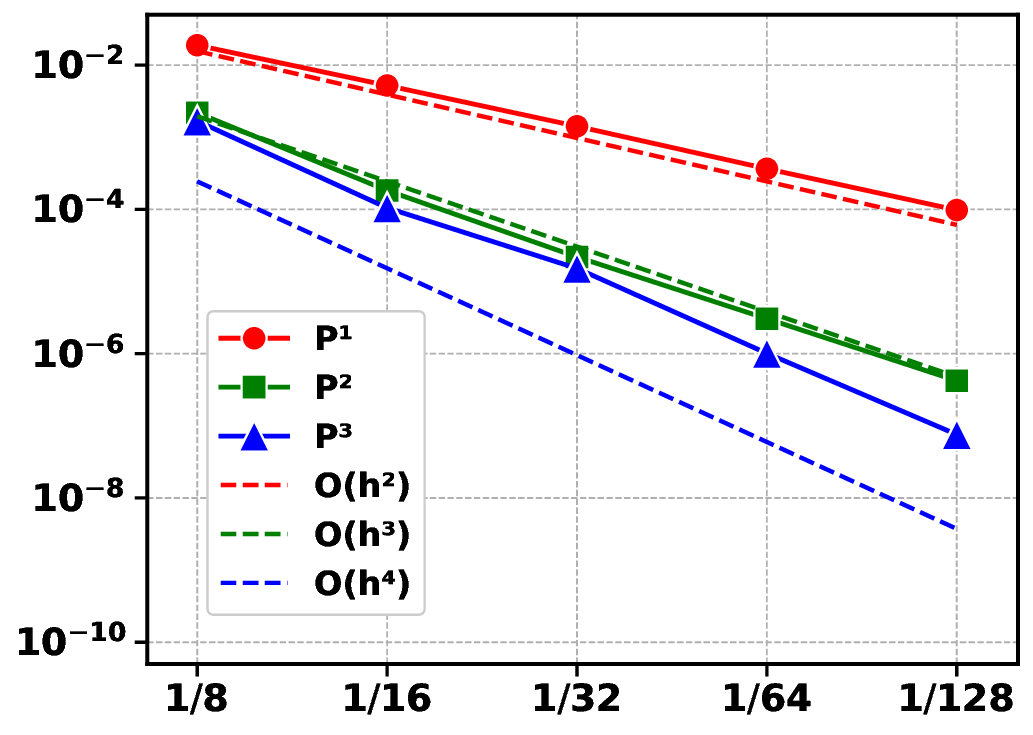}
        \caption{Results at time $T=0.5$ from our proposed  CutDG schemes with different polynomials solving the 2D linear advection equation in Example \ref{test: flower}.  From left to right: one illustration of mesh;  $P^3$ solution with mesh size $h=1/128$;  L$^2$ errors; L$^{\infty}$ errors.}
        \label{fig: linear 2D result}
    \end{center}
\end{figure}

\vspace{-1.5em}
\subsubsection{A three dimensional problem}
We now consider the three-dimensional problem \eqref{eq:linear} with constant coefficients $\be=(2,2,2)$ in $\F(u)=\be u$. The physical domain is a sphere $\Omega=\{ (x,y,z) \in \Omega: x^2+y^2+z^2 \leq c_0\}$, with $c_0=1$. The physical surface is $\Gamma: \{ (x,y,z) \in \Omega: x^2+y^2+z^2=1\}$. We embed the domain $\Omega$ in a cube $[-1.1,1.1]\times[-1.1,1.1]\times[-1.1,1.1]$ partitioned into regular uniform cubes.  
We take the initial condition to be $u_0(x,y,z)=0.5+\sin(0.5 \pi (x+y+z))$ and and use Dirichlet boundary condition such that the solution is
$
    u_b = 0.5+\sin(0.5 \pi (x+y+z - 6t)).
$ 
We compute this problem up to $T=0.5$. In the left Figure~\ref{fig: linear 3D result}, the $P^2$ solution on the mesh with $h=2.2/80$ is shown. Note that there are no obvious oscillations. In the right panel of  Figure~\ref{fig: linear 3D result}, $L^2$ and $L^\infty$ errors are plotted, respectively. 
Optimal convergence is again observed for both $P^1$ and $P^2$ approximations. 
\begin{figure}[!htb]
    \begin{center}
        \includegraphics[width=0.25\textwidth]{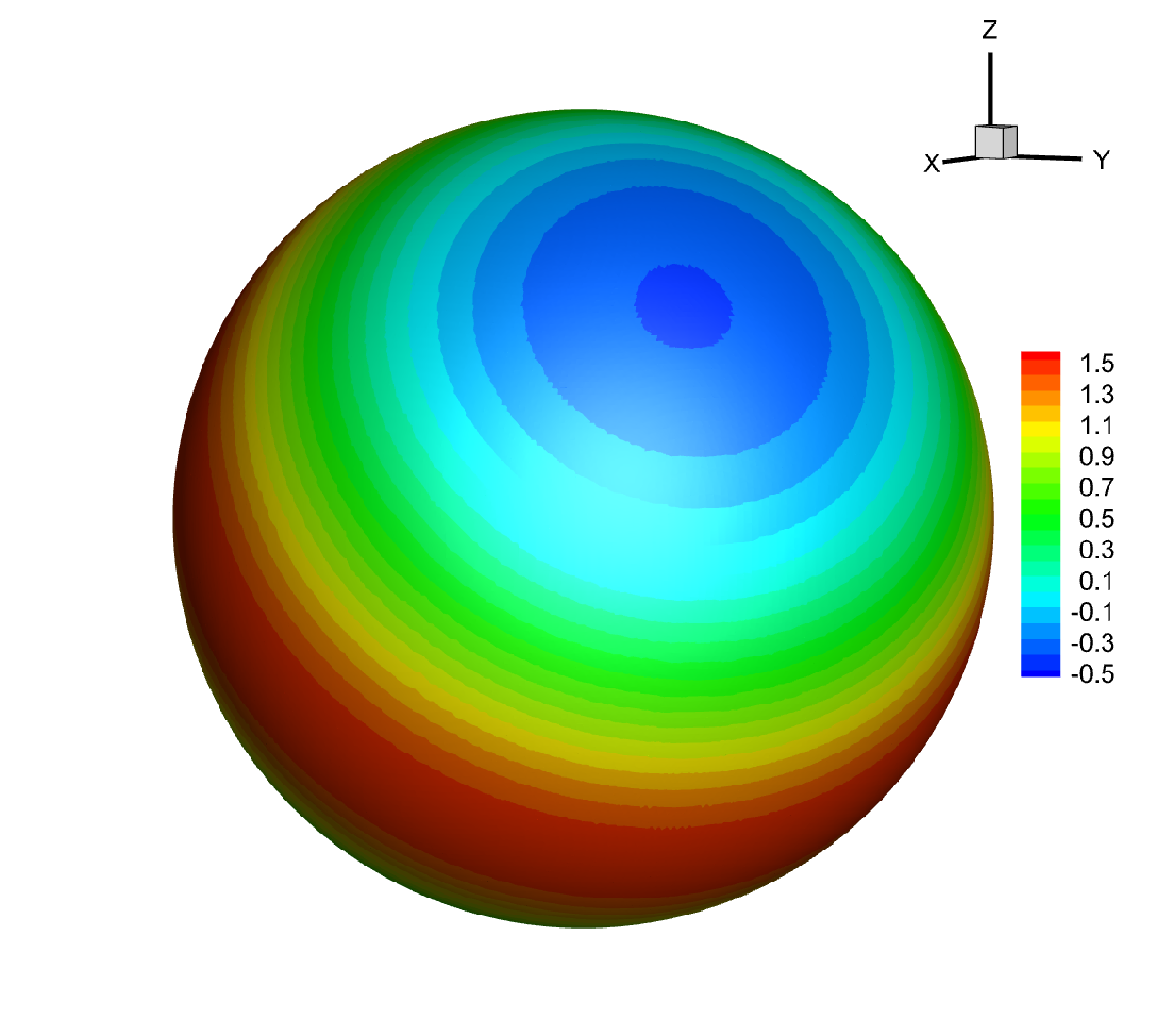} \quad
        \includegraphics[width=0.28\textwidth]{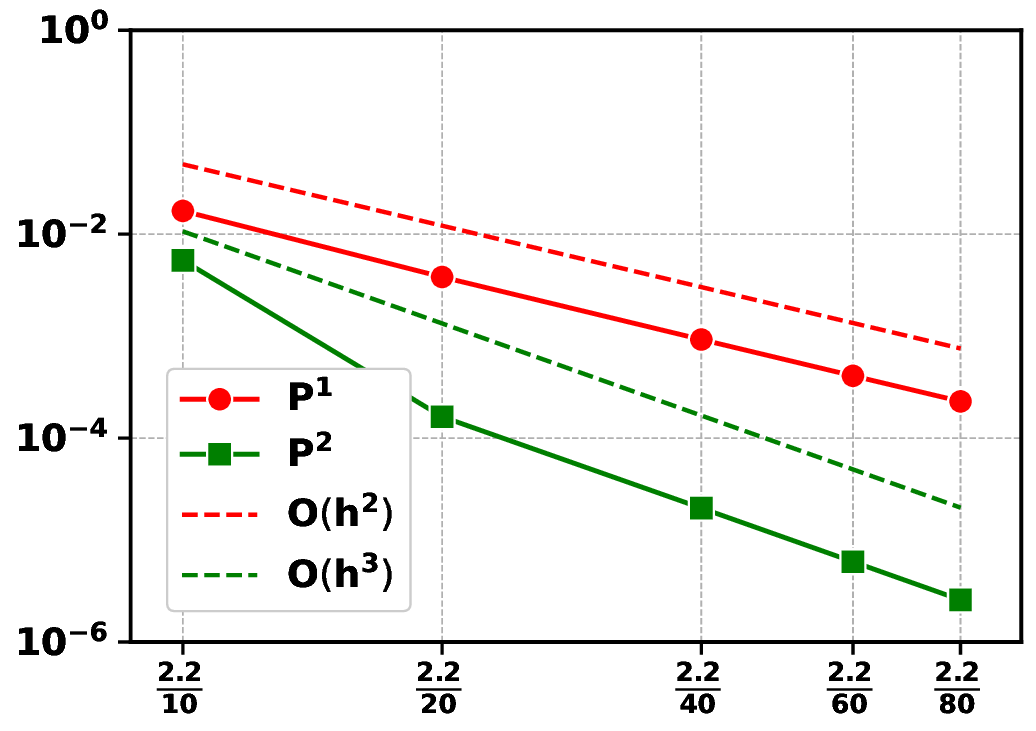} \quad
        \includegraphics[width=0.28\textwidth]{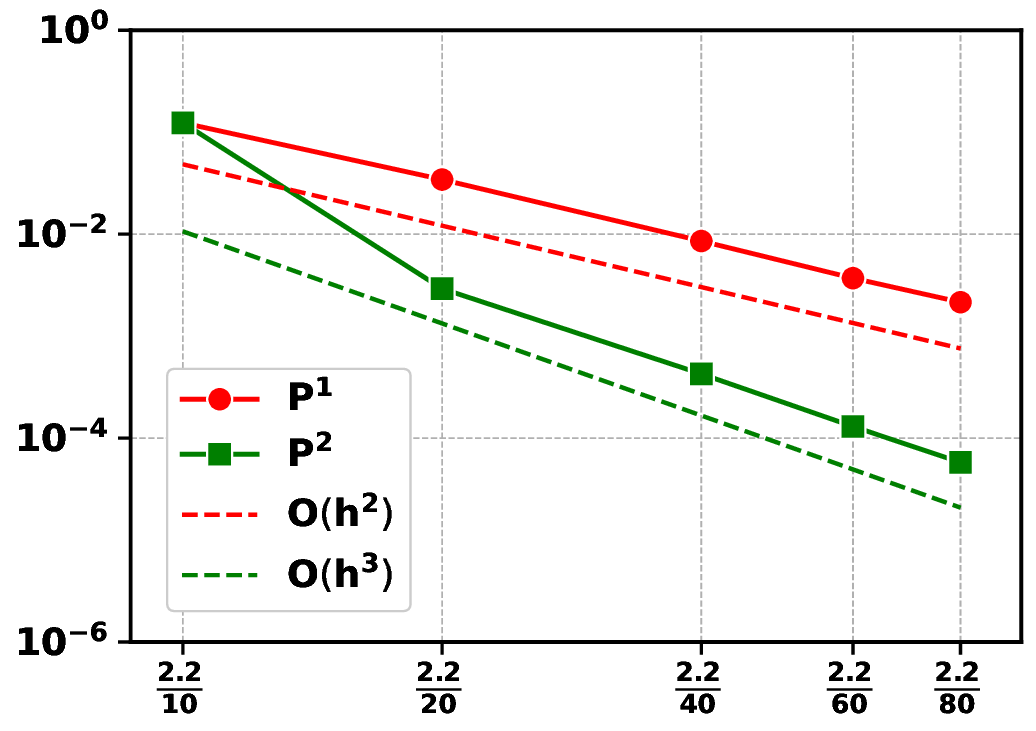}
        \caption{Results at $T=0.5$ from the proposed CutDG schemes solving the 3D linear advection equation.  Left: $P^2$ solutions on the mesh $h=2.2/80$. Middle: $L^2$ errors. Right: L$^{\infty}$ errors.}
        \label{fig: linear 3D result}
    \end{center}
\end{figure}

\subsubsection{A two dimensional non-smooth problem}
We next consider the linear advection problem \eqref{eq:linear} with $\be=(1,1)$ and non-smooth initial data. 
Let $\mt_{B_h}$ be uniform partition of the domain $[0, 2] \times [0, 2]$. Define the physical boundary $\Gamma$ by  the line $y = x - c_0$ with $c_0 $ being a constant. The physical domain is set as $\Omega=\{ (x,y) \in \mt_{B_h}: y \geq x-c_0\}$. In our computation, we set $c_0 = 0.5001$.  The following non-smooth initial data is considered
\begin{align}
    u_0(x,y)= & \left\{ \begin{array}{ll}
        1.5 & \mbox{if $1<x+y<3$,}\\
        -0.5  & \mbox{else}.
    \end{array} \right.    \nonumber
\end{align}
We solve this problem on a uniform background mesh with mesh size $h=1/40$ up to $T=1$. For this problem, the modified BJ limiter from Section \ref{sec:slopelimiter} is used to control oscillations. 
The results with $P^3$ approximation are shown in Figure~\ref{fig:step_result}. The method captures the discontinuities while preserving the maximum principle for macro-element mean values. To quantify maximum-principle preservation, we define 
\begin{align}
e^\MM(t^n)=\uM- \max_{M\in\mcM_h } \uhmo,\quad e^{\mm}(t^n)=\um- \min_{M\in\mcM_h} \uhmo.
\end{align}
 The maximum principle is satisfied whenever \(e^\MM\ge 0\) and \(e^{\mm}\le 0\). Figure~\ref{fig:step_result} shows that \(\uhmo\) remains within  \([\um,\uM]\) up to machine precision. 
\begin{figure}[!htb]
    \begin{center}
        \includegraphics[width=0.25\textwidth]{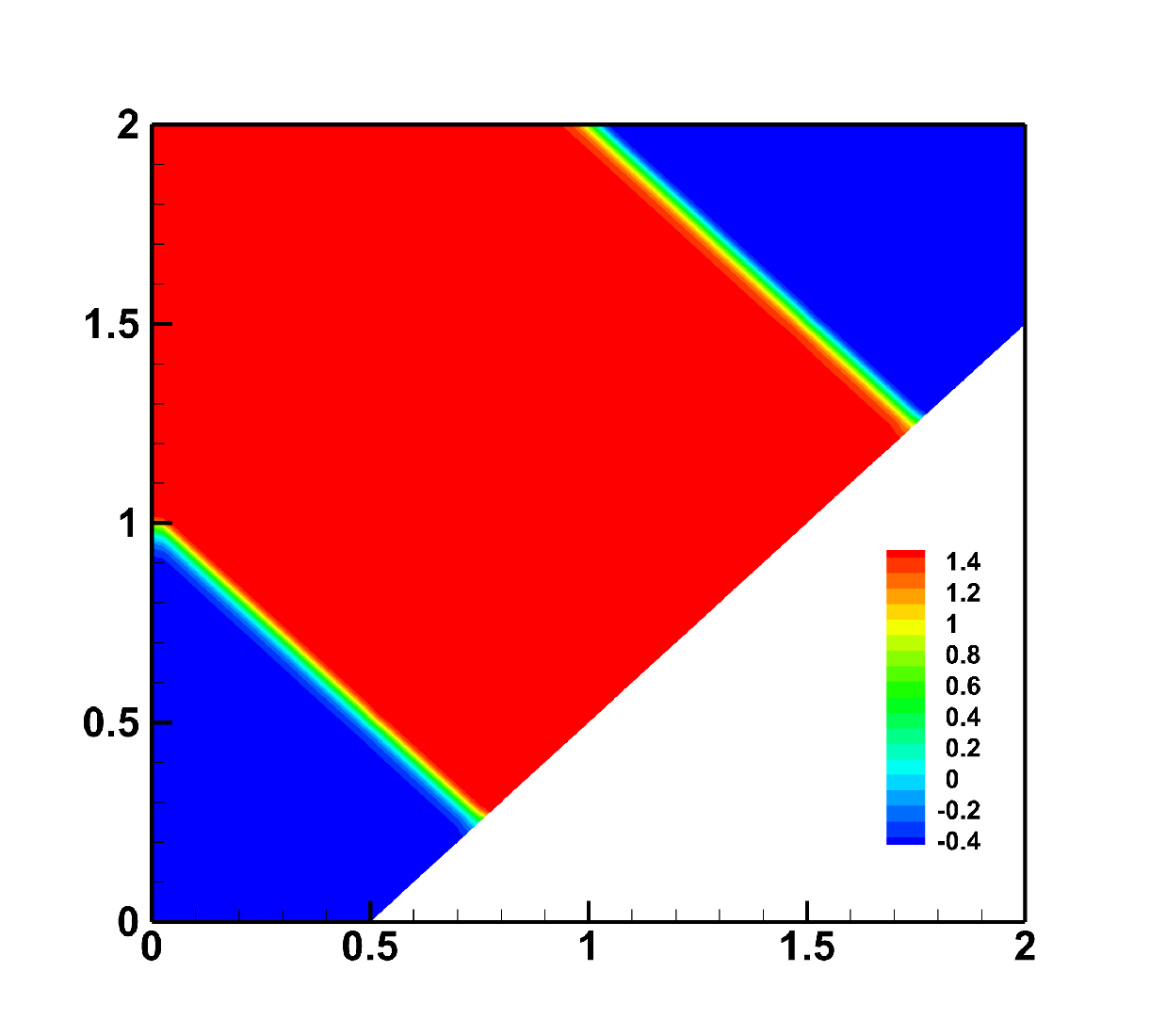} \quad
        \includegraphics[width=0.25\textwidth]{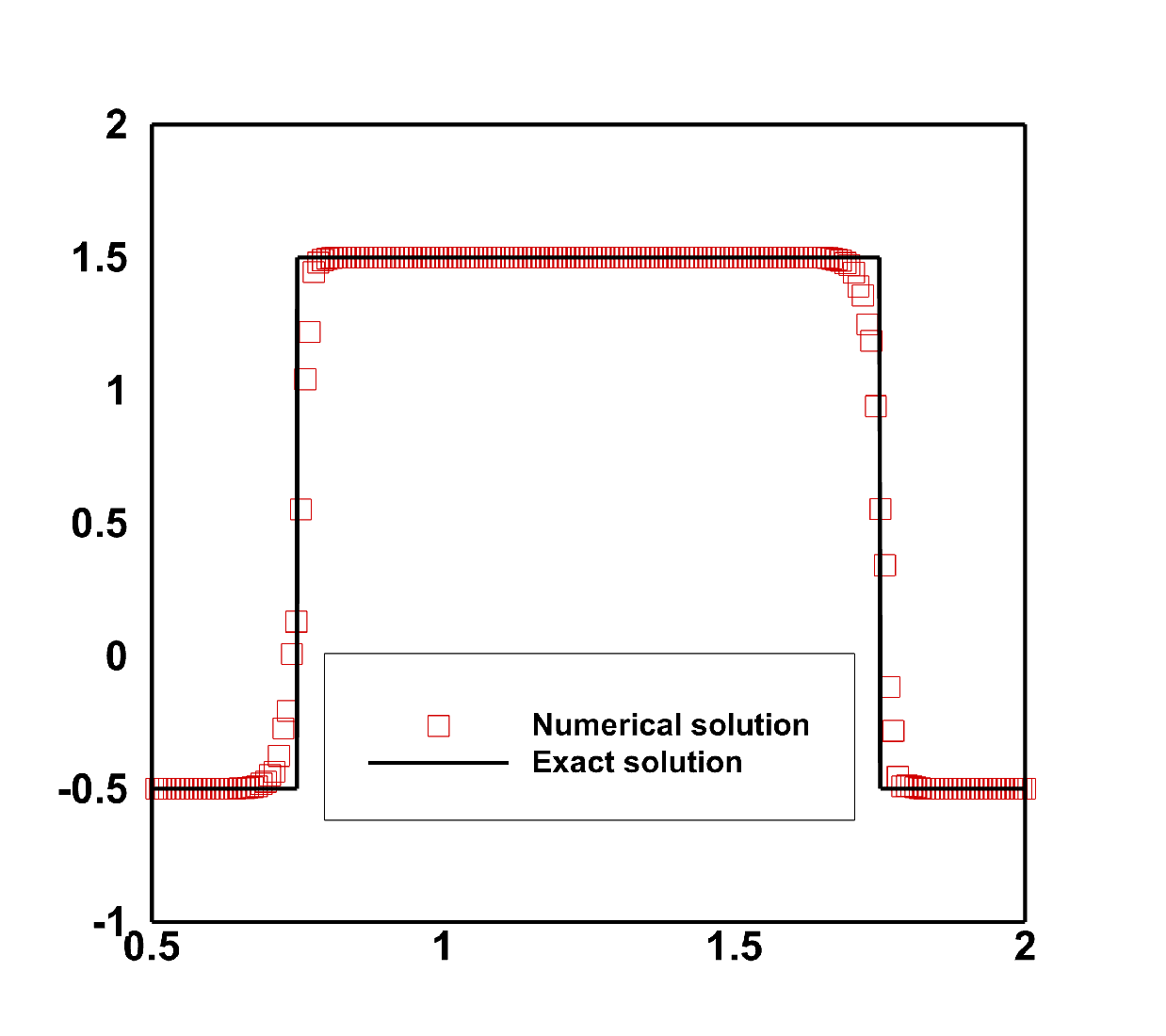} \quad
        \includegraphics[width=0.25\textwidth,height=0.20\textwidth]{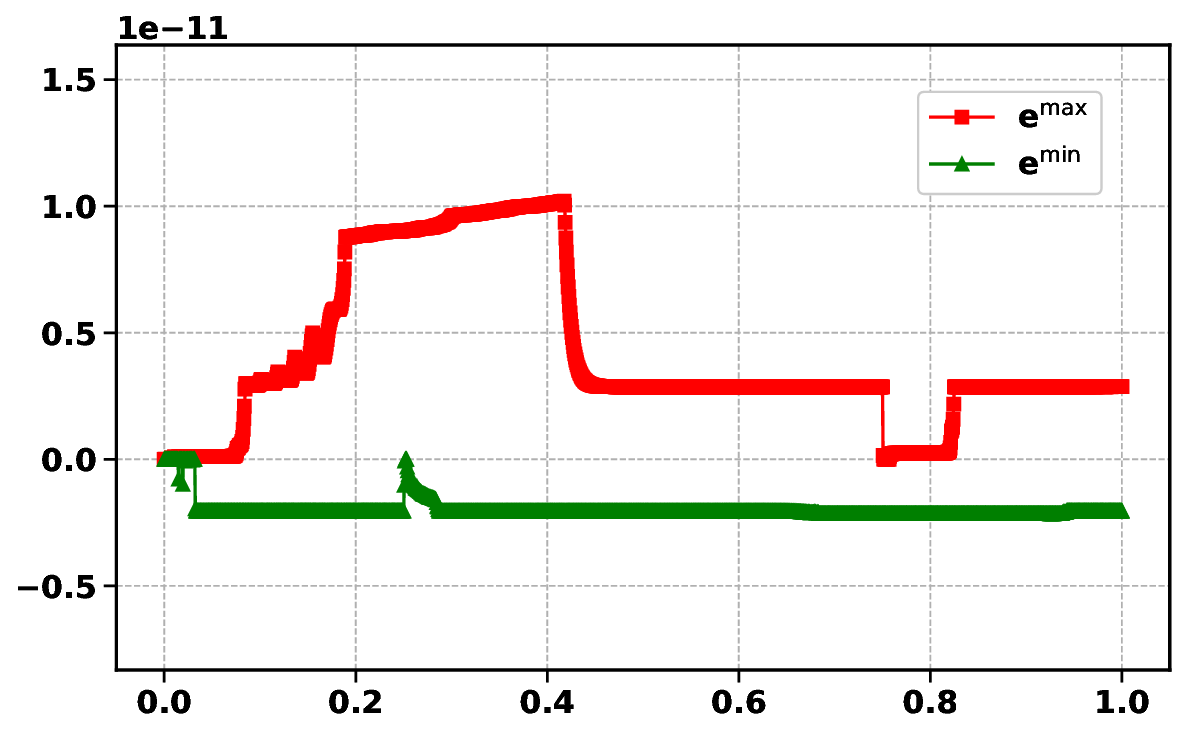} 
        \caption{From left to right: 
            (1) Solution for the linear advection equation with non-smooth initial data, where both the MPP flux limiter and BJ limiter are applied.
            (2) Solutions on~$\Gamma$. 
               (3) MPP property: $e^{\mm}$ (green) and $e^\MM$ (red) versus time~$t$. 
            }
        \label{fig:step_result}
    \end{center}
\end{figure}
\vspace{-1.3em}

For comparison, we repeat the computation without the MPP flux limiter.  
The results are shown in Figure~\ref{fig: step function with barth limiter}.  Although oscillations are largely suppressed, the maximum principle is no longer guaranteed. 
\begin{figure}[!htb]
    \begin{center}
        \includegraphics[width=0.35\textwidth]{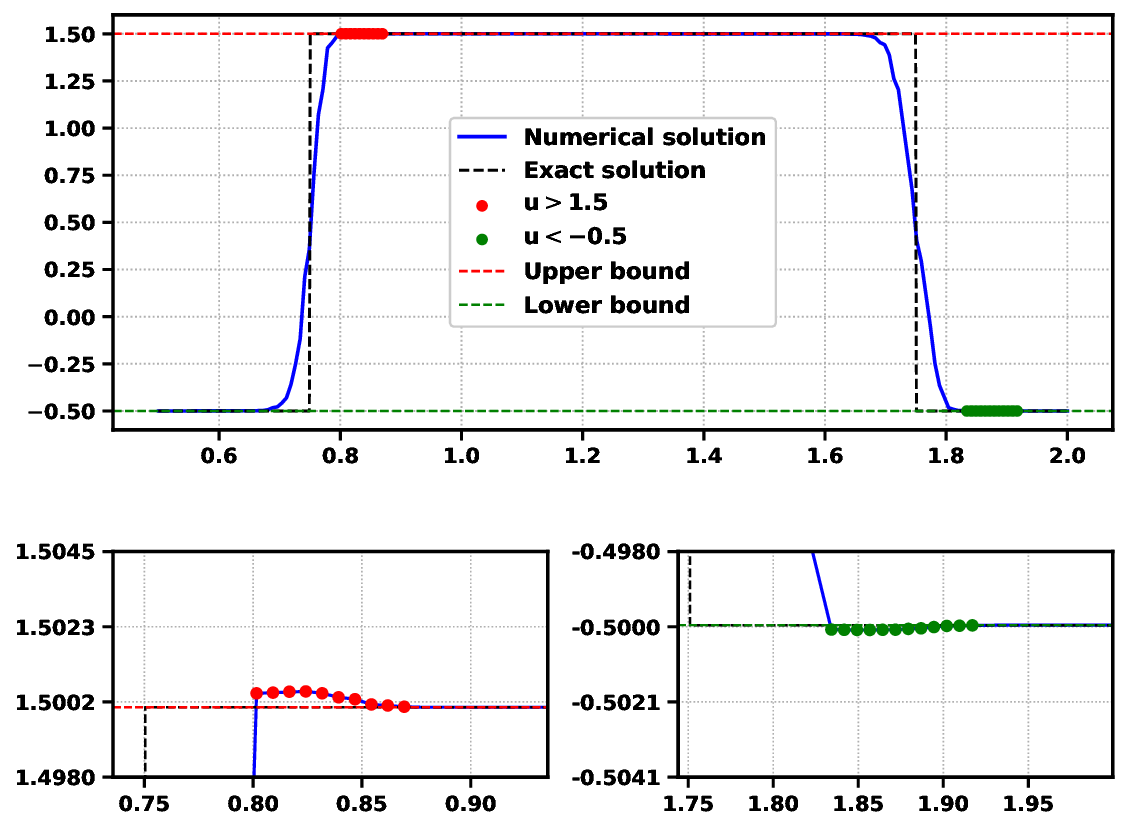}
        \includegraphics[width=0.4\textwidth]{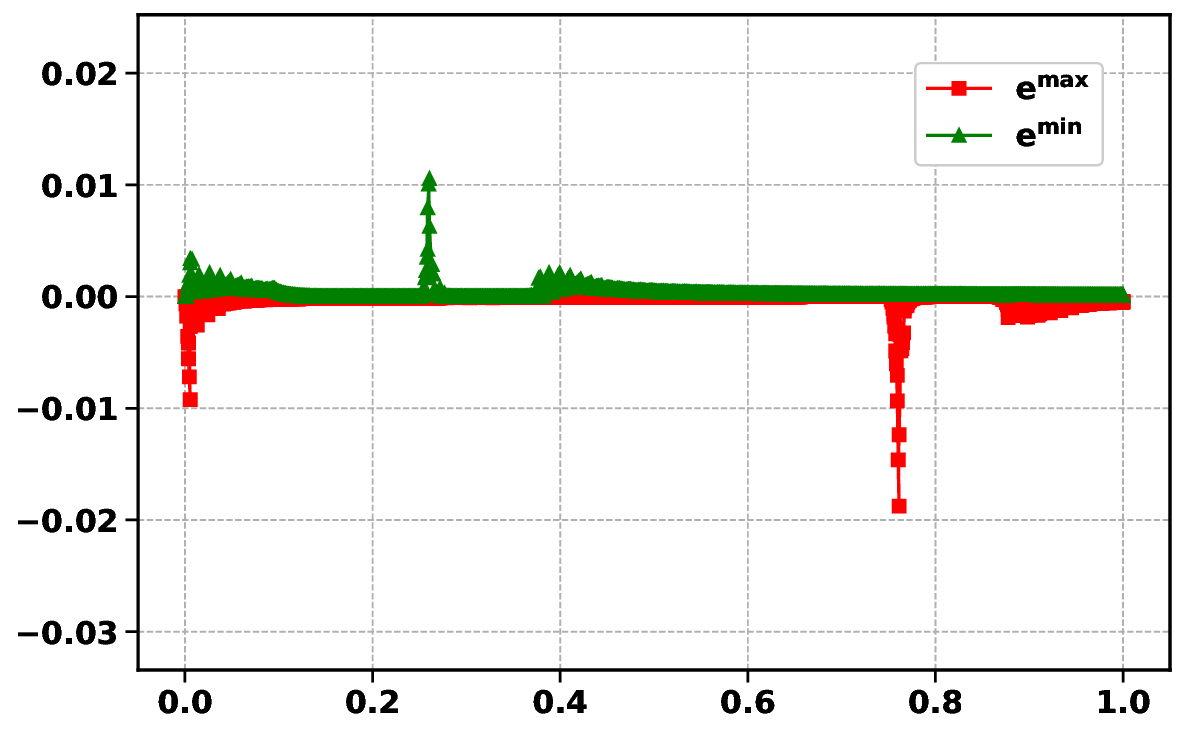}
        \caption{Results from the scheme without MPP flux limiter for the advection equation with non-smooth initial data.  Left: $P^1$ solution along the cut line $\Gamma$ from the scheme \eqref{eq:def:scheme} with only the BJ limiter.
            Right:
            Temporal evolution of $e^{\mm}$ (green) and $e^\MM$ (red).
        }
        \label{fig: step function with barth limiter}
    \end{center}
\end{figure}

\subsection{Nonlinear problems: Burgers' equation}\label{sec:num:nonlinear}
In this subsection, we solve the two-dimensional Burgers' equation
$
	u_{t}+\left(\frac{u^{2}}{2}\right)_{x}+\left(\frac{u^{2}}{2}\right)_{y}=0.
$ 
We consider the problem on a circular domain $\Omega=\{(x,y):x^2+y^2\leq4\}$, which is embedded in the uniform partitioned domain $[-2,2]\times[-2,2]$. 
The initial condition is $u_0(x,y)=0.5+\sin (\pi(x+y))$. In this setting, the shock appears when $t>\frac{1}{\pi}$.

We first solve this problem by  our proposed CutDG method with different polynomials $P^k,k=1,2,3$ until  $T=\frac{0.5}{\pi}$ before shock shows.  The resulting  $L^2$ and $L^\infty$ errors at the final time are shown in Figure  \ref{fig:burgers result for smooth}, which indicates optimal accuracy of our proposed CutDG method also for nonlinear problems.
\begin{figure}[!htb]
	\begin{center}
		\includegraphics[width=0.3\textwidth]{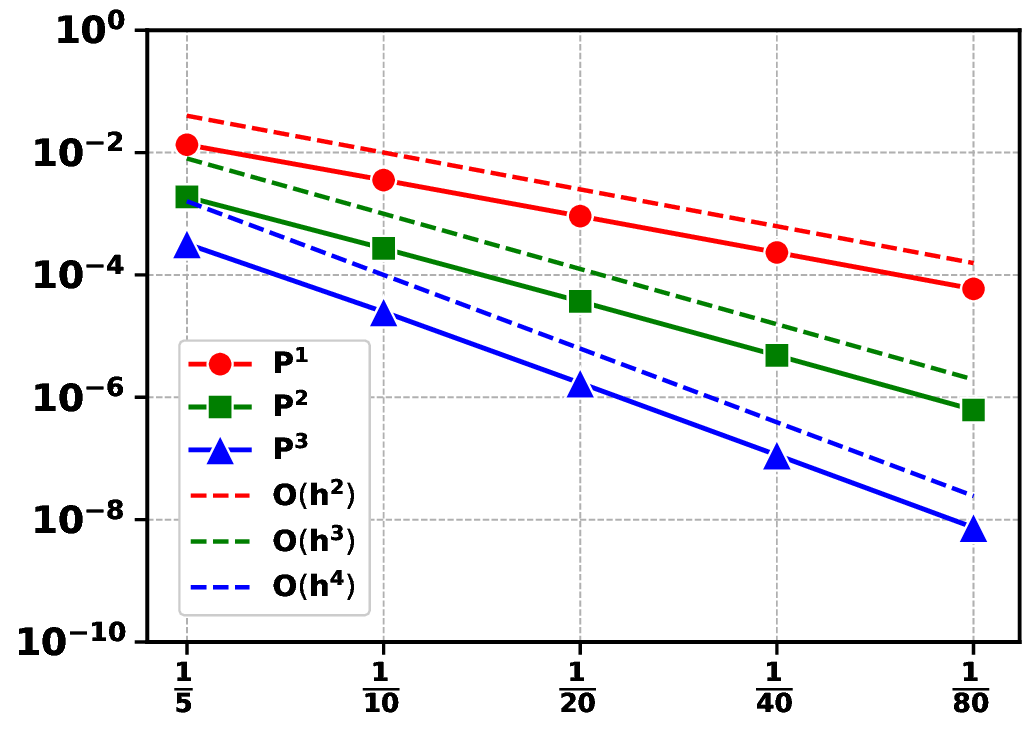} \qquad\qquad
		\includegraphics[width=0.3\textwidth]{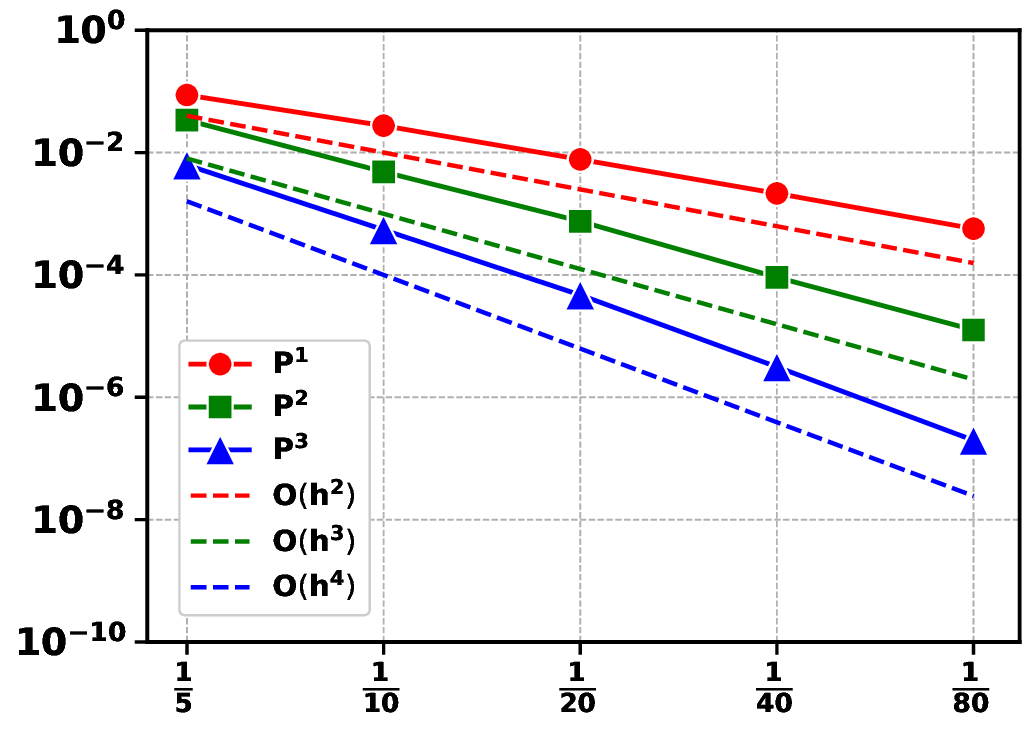}
		\caption{Errors  of the Burgers' equation from our proposed method at $T=0.5/\pi$ with $P^k$, $k=1,2,3$ approximations.  
		Left: L$^2$ errors, Right: L$^{\infty}$ errors.}
		\label{fig:burgers result for smooth}
	\end{center}
\end{figure}

Next we solve this problem on the background mesh with $h=1/20$ and $P^3$ approximation up to time $T=\frac{1.5}{\pi}$ when a shock has formed.  The BJ limiter \eqref{eq:slopelimit:uh} is applied in each stage of RK time step to suppress the oscillations near the shock.  The numerical results are shown in Figure \ref{fig:burgersd_result}, which shows that our proposed method can capture shocks well, while maintaining the maximum principle.
\begin{figure}[!htb]
	\begin{center}
		\includegraphics[width=0.25\textwidth]{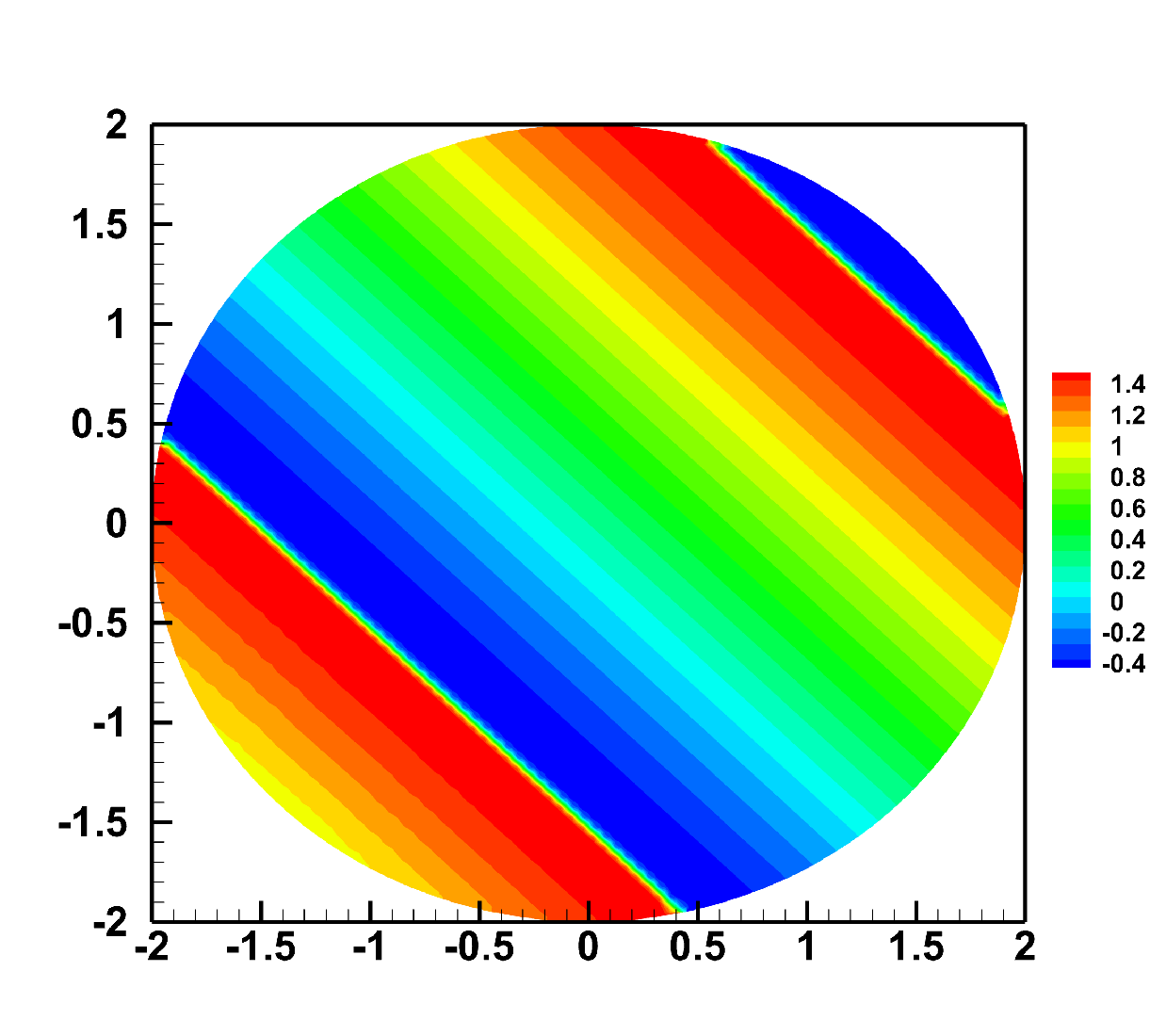} \quad
		\includegraphics[width=0.25\textwidth]{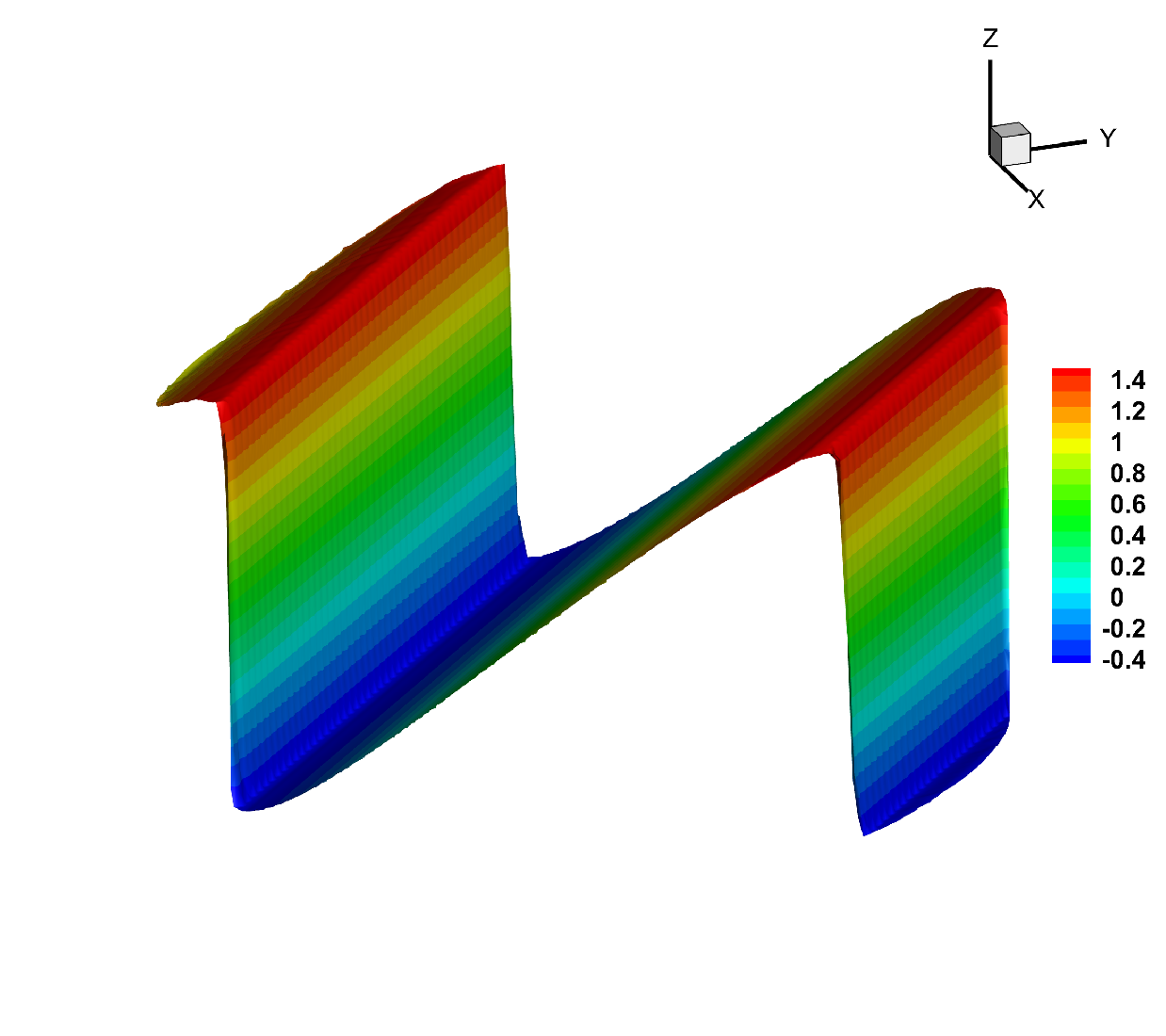} \quad
		\includegraphics[width=0.25\textwidth,height=0.20\textwidth]{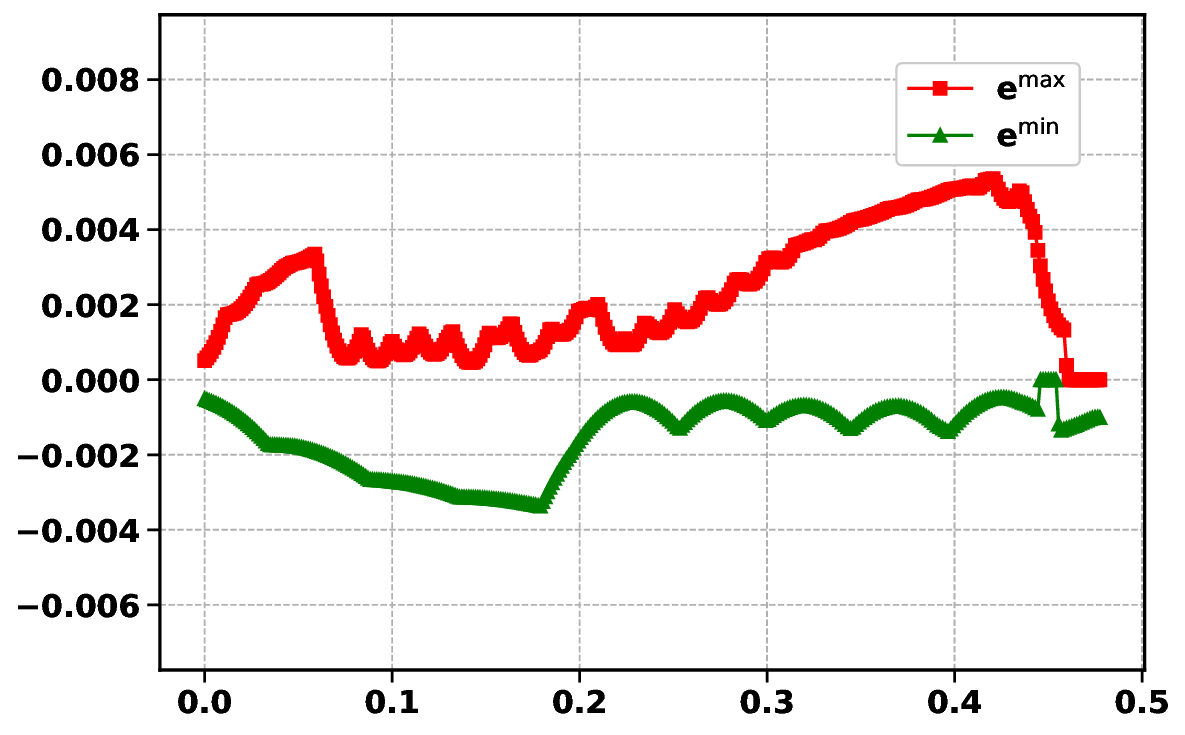}
		\caption{Results of the Burgers' equation at $T=1.5/\pi$ from our proposed method. Left: $P^3$ solution in 2D plot. Middle: 3D plot.  Right: MPP: $e^{\mm}$ and $e^\MM$ versus time~$t$.  
		}
		\label{fig:burgersd_result}
	\end{center}
\end{figure}

For comparison, we also present the numerical values on the Gaussian quadrature points for the $P^3$ approximation with different combinations of limiters at time $T=\frac{1.5}{\pi}$. Comparing the left panel of Figure~\ref{fig:burgersd_result_polylimit} with the right panel of Figure~\ref{fig:burgersd_result}, we observe that the MPP flux limiter preserves the bounds for the mean values of macro-elements  $\uhm\in[\um,\uM]$, whereas the point values of the DG polynomial may still violate the MPP property. The middle figure of Figure~\ref{fig:burgersd_result_polylimit} shows the results obtained by applying both the MPP flux limiter and the bound-preserving polynomial limiter only at the final stage of each Runge-Kutta step. In this setting, the method captures the shock sharply and keeps both the mean values $\uhm$ and the values of $u_h$ on the Gaussian quadrature points within the admissible interval $[\um,\uM]$. The right figure of Figure~\ref{fig:burgersd_result_polylimit} presents the results obtained by the $P^3$ approximation with all three limiters. In this case, both $\uhm$ and the quadrature-point values of $u_h$ remain within $[\um,\uM]$. However, the limiters introduce a more restrictive correction, leading to a slightly more diffusive method.
\begin{figure}[!htb]
	\begin{center}
		\includegraphics[width=0.26\textwidth,height=0.22\textwidth]{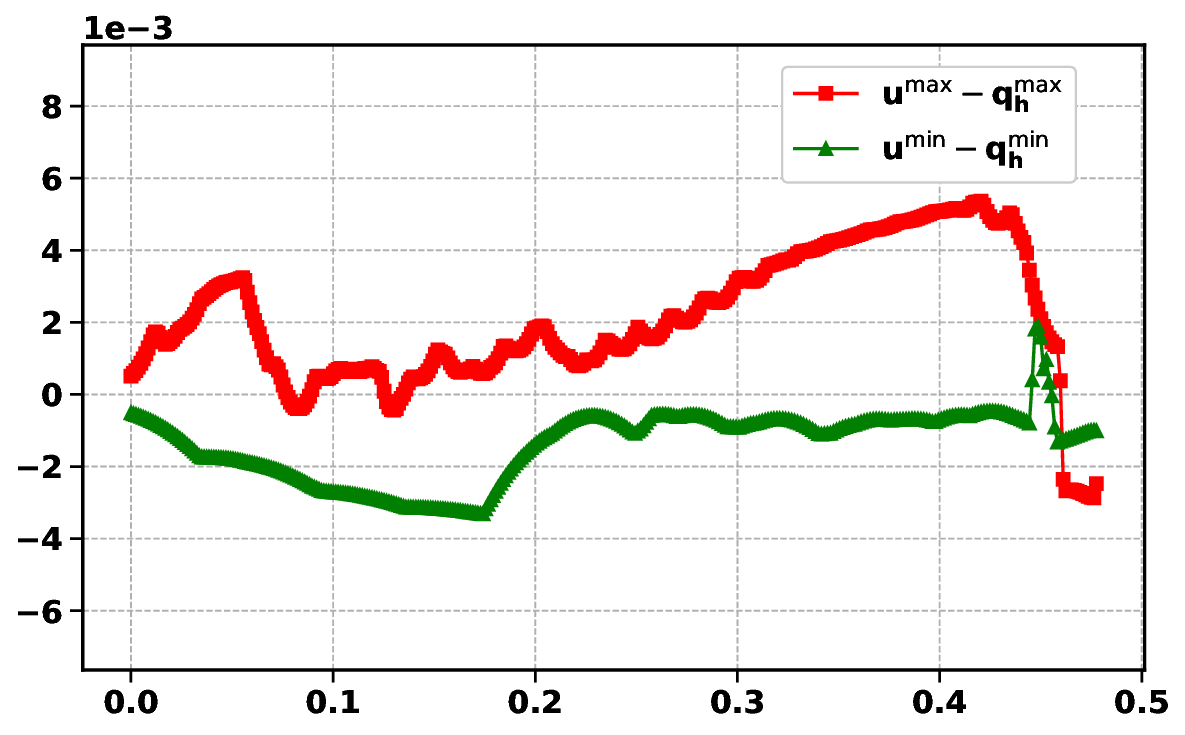}
		\includegraphics[width=0.26\textwidth,height=0.22\textwidth]{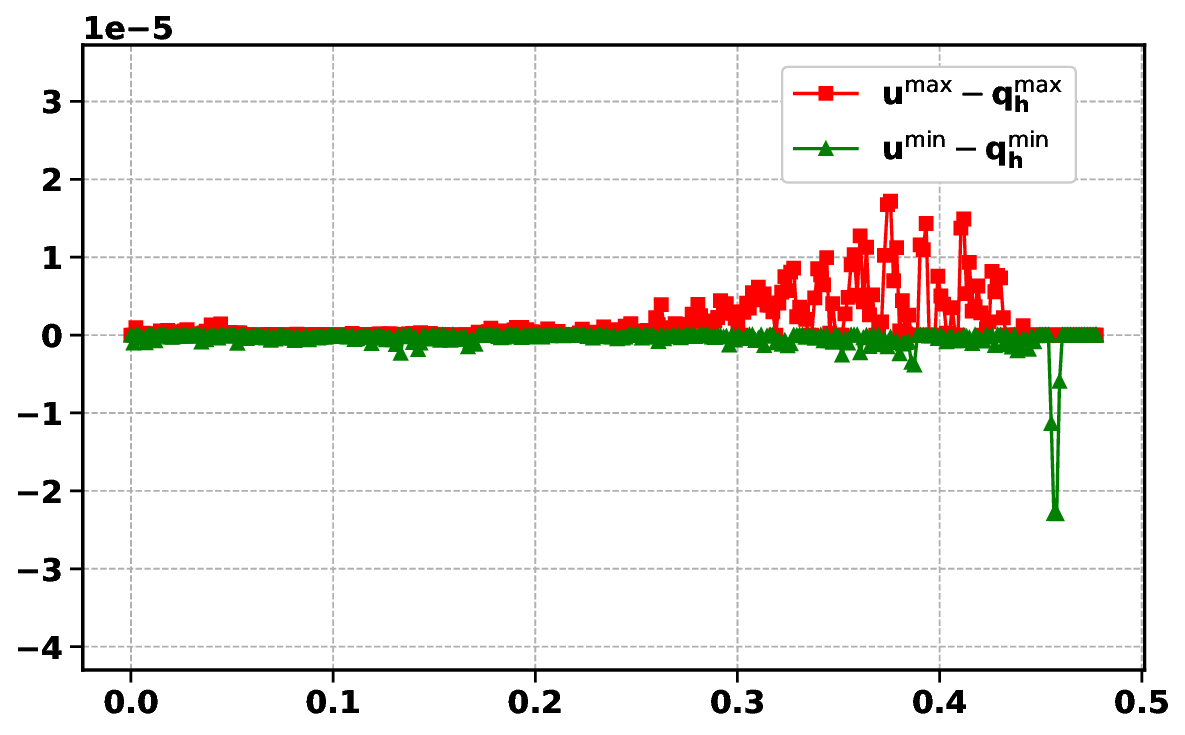}
		\includegraphics[width=0.26\textwidth,height=0.22\textwidth]{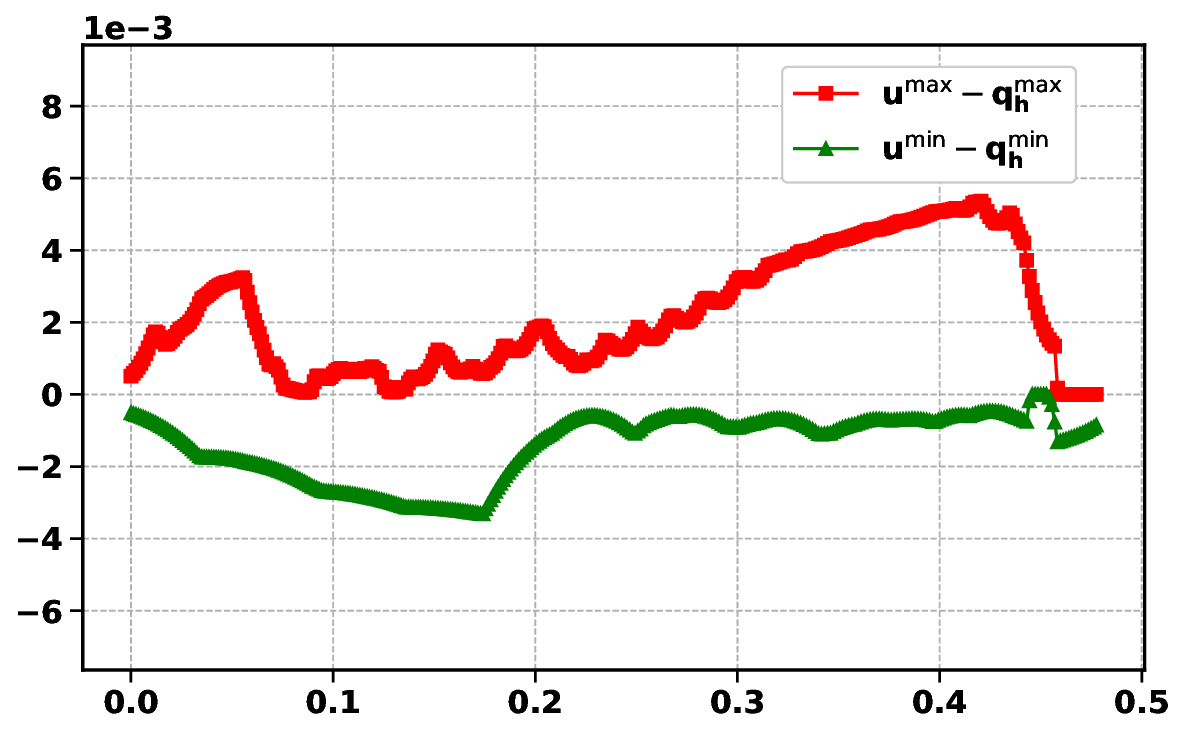}
		\caption{The difference of $u^{\mathrm{min}}-q_h^{\mathrm{min}}$ and $u^{\mathrm{max}}-q_h^{\mathrm{max}}$ from $P^3$ solution at $T=1.5/\pi$ with different combinations of limiters. Left: flux limiter with the Barth Jespersen limiter. Middle: flux limiter with the bound-preserving limiter. Right: flux limiter with both the bound preserving and Barth Jespersen limiters.
		}
		\label{fig:burgersd_result_polylimit}
	\end{center}
\end{figure}

Next, we also test Burgers' equation by the proposed method but using the flux splitting formulation \eqref{scheme:2D:ah:split}.  The $L^2$ and $L^\infty$ errors are shown in Figure \ref{fig:burgers result for smooth with flux splitting} which indicates this formulation with our proposed MPP flux limiting strategy can also get the expected optimal convergence and MPP property for the mean values. The numerical results show that this method yield nearly identical solutions with our proposed method, indicating that they are highly consistent with each other and the numerical quadrature is sufficiently accurate in our computation.
\begin{figure}[!htb]
	\begin{center}
		\includegraphics[width=0.3\textwidth]{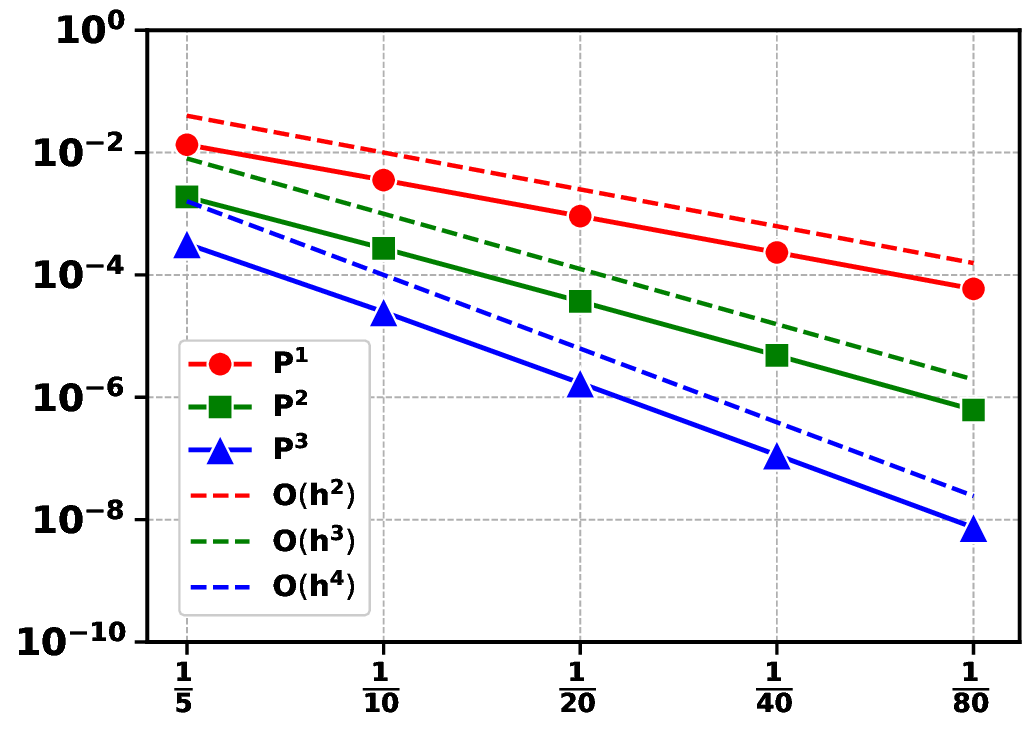} \qquad\qquad
		\includegraphics[width=0.3\textwidth]{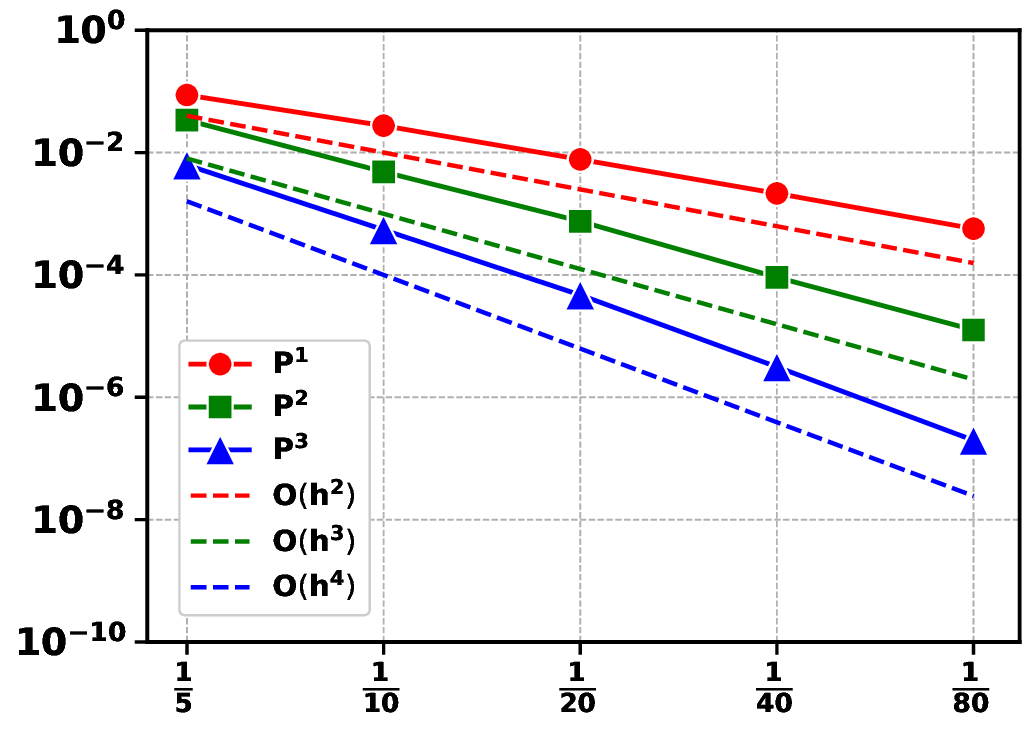}
		\caption{Errors and convergence orders for the $P^1$, $P^2$, and $P^3$ CutDG schemes applied to the Burgers' equation. The flux limiter is employed, together with a flux-splitting formulation \eqref{scheme:2D:ah:split}. Left: L$^2$ errors, Right: L$^{\infty}$ errors.}
		\label{fig:burgers result for smooth with flux splitting}
	\end{center}
\end{figure}
\begin{figure}[!htb]
	\begin{center}
		\includegraphics[width=0.3\textwidth]{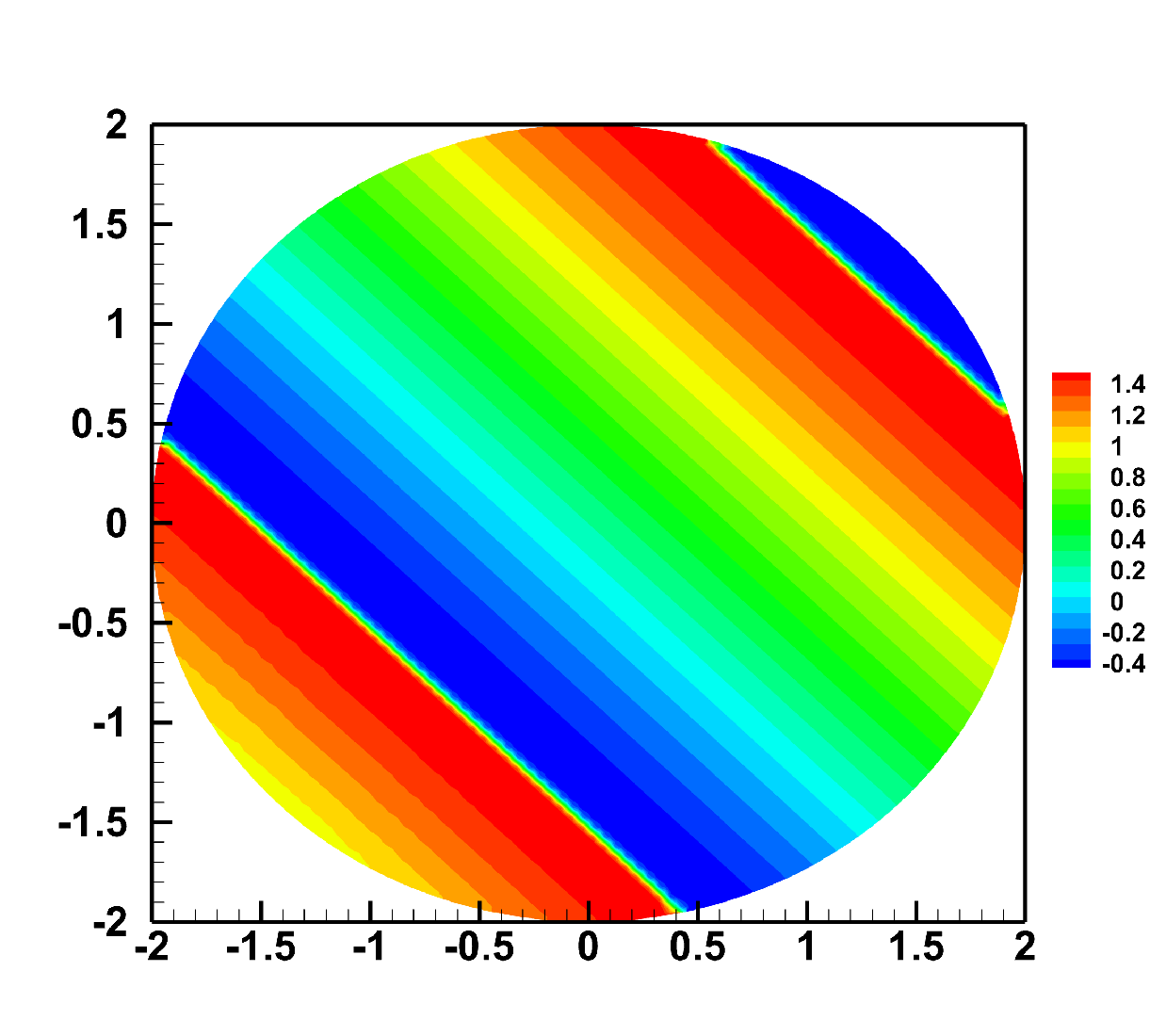}
		\includegraphics[width=0.3\textwidth]{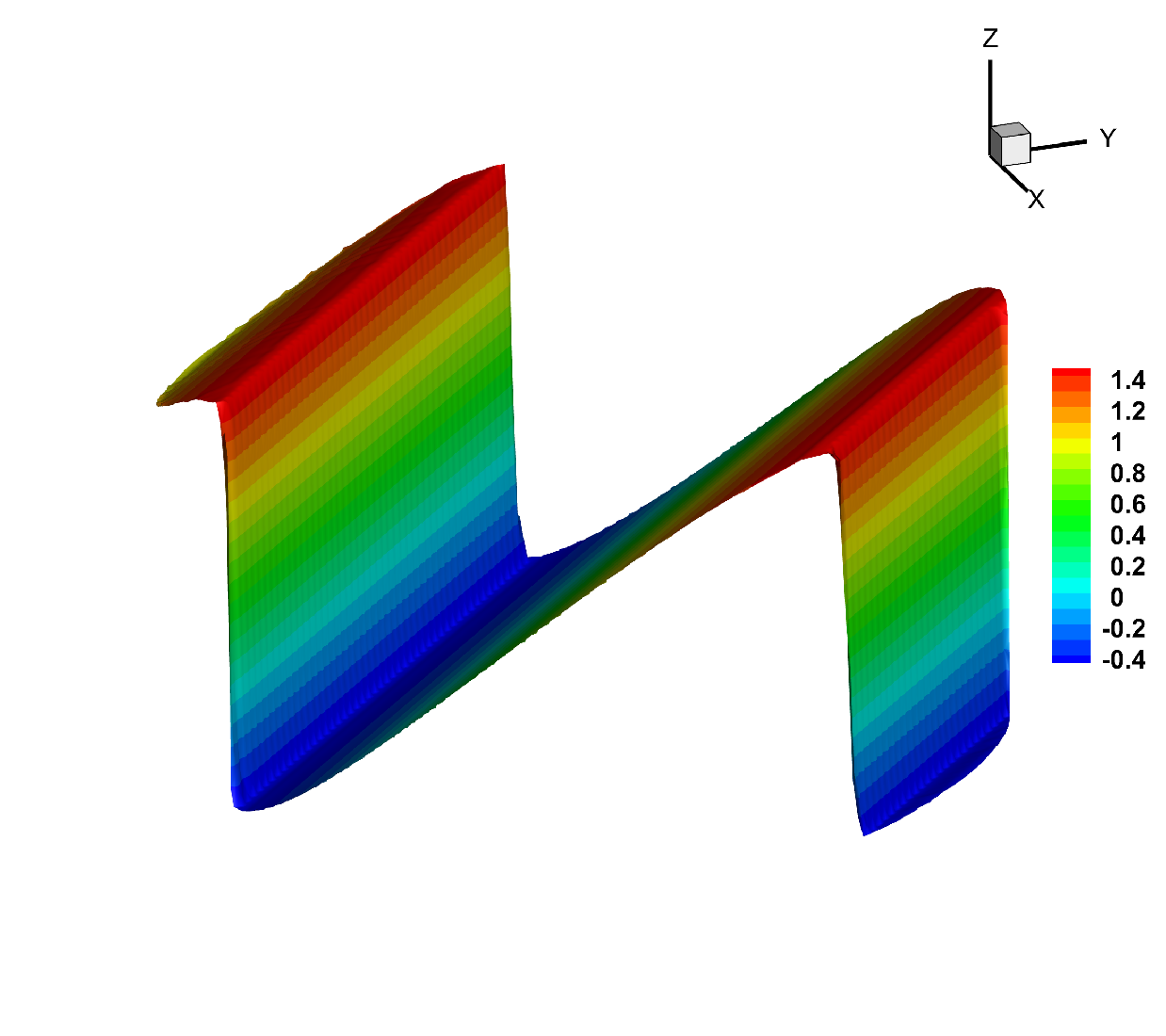}
		\includegraphics[width=0.3\textwidth,height=0.24\textwidth]{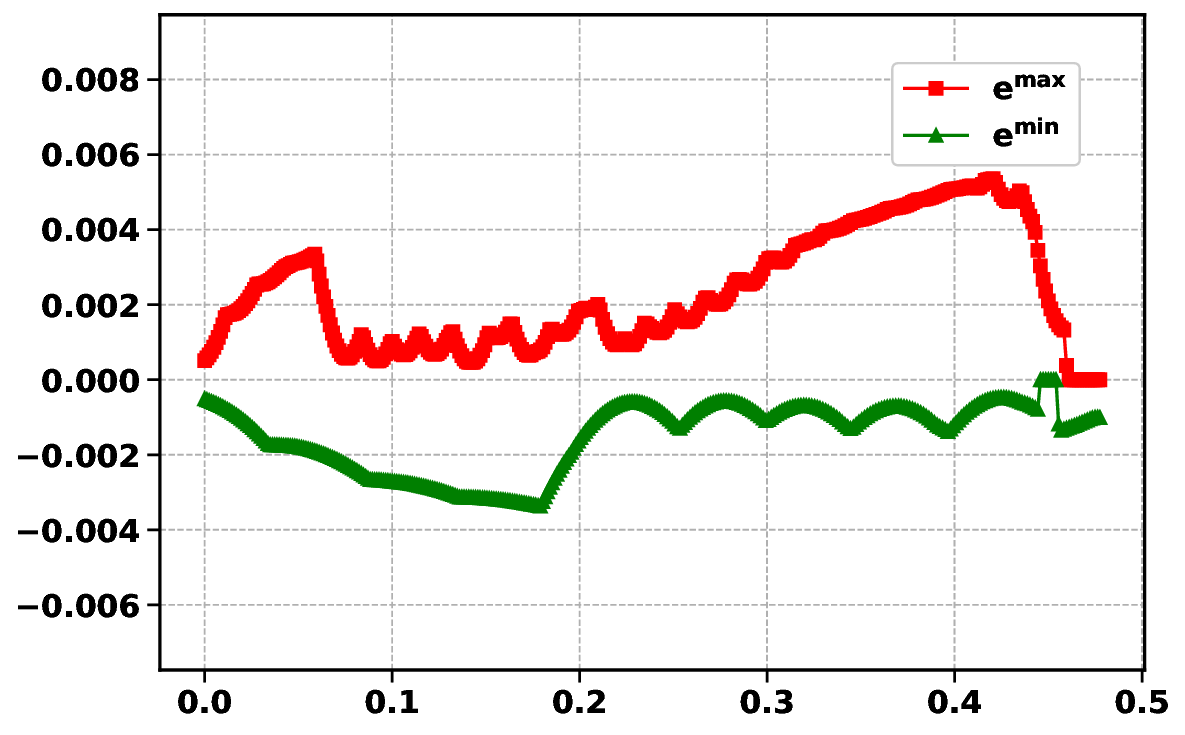}
		\caption{$P^3$ solution of the Burgers' equation at $T=1.5/\pi$ from the scheme with bilinear form  \eqref{scheme:2D:ah:split}. The flux limiter and the BJ limiter are used. Left: solution in 2D plot; middle: solution in 3D plot; right: $e^{\mm}$ and $e^\MM$ versus time~$t$ to show MPP property.
		}
		\label{fig:burgersd_result_barth_split}
	\end{center}
\end{figure}

\section{Conclusion}
\label{sec:conclusion}
We have presented a family of high-order CutDG methods for scalar hyperbolic conservation laws on complex domains in two and three space dimensions. A central contribution of this work is the extension of limiter techniques from standard DG methods to the CutFEM setting through the use of macro-elements. In particular, we introduced a macro-element-based flux-limiting strategy and adapted the Zhang-Shu bound-preserving limiter and the Barth--Jespersen slope limiter to unfitted meshes. The resulting fully discrete scheme preserves the maximum principle and suppresses spurious oscillations near discontinuities. For the semi-discrete scheme, we established an $L^2$-stability result for both periodic and inflow--outflow boundary conditions. Numerical experiments in two and three space dimensions demonstrated optimal convergence rates for smooth solutions and robust shock-capturing performance for discontinuous problems. In particular, the time-step restriction remains comparable to that of the corresponding fitted DG method and is essentially independent of the cut configuration. The proposed methodology is also applicable to systems of conservation laws. Extending the framework to the Euler equations is the subject of ongoing work. 

\appendix

\section*{Acknowledgments}
Pei Fu was partially supported by the National Natural Science Foundation of China No. 12301470 and by the Natural Science Foundation of Jiangsu Province No. BK20230869, the Fundamental Research Funds for the Central Universities, China, No.25024. Sara Zahedi was supported by Wallenberg Academy Fellow Grants KAW 2019.0190 and KAW 2024.0234.  This material is based upon work supported by the Swedish Research Council under Grant No. 2021-06594 while the authors were at Institut Mittag-Leffler, Djursholm, Sweden, during the Fall 2025 program.

\bibliographystyle{siam}
\bibliography{CutDG}

@article{LarZah23,
	Author = {Mats G. Larson and Sara Zahedi},
	Journal = {Comput. Methods Appl. Mech. Engrg.},
	Pages = {116141},
	Title = {Conservative cut finite element methods using macroelements},
	Volume = {414},
	Year = {2023}
	}

@article{shu2009discontinuous,
  title={{Discontinuous Galerkin methods: general approach and stability}},
  author={Shu, Chi-Wang},
  journal={Numerical solutions of partial differential equations},
  pages={149--201},
  year={2009},
  publisher={Birkh{\"a}user Basel}
}

@article{ShuDG4,
  title={{The Runge-Kutta local projection discontinuous Galerkin finite element method for conservation laws. IV. The multidimensional case}},
  author={Cockburn, Bernardo and Hou, Suchung and Shu, Chi-Wang},
  journal={Mathematics of Computation},
  volume={54},
  number={190},
  pages={545--581},
  year={1990}
}

@article{jiang1994cell,
  title={{On a cell entropy inequality for discontinuous Galerkin methods}},
  author={Jiang, Guang Shan and Shu, Chi-Wang},
  journal={Mathematics of Computation},
  volume={62},
  number={206},
  pages={531--538},
  year={1994}
}

@article{colella2006cartesian,
  title={A Cartesian grid embedded boundary method for hyperbolic conservation laws},
  author={Colella, Phillip and Graves, Daniel T and Keen, Benjamin J and Modiano, David},
  journal={Journal of Computational Physics},
  volume={211},
  number={1},
  pages={347--366},
  year={2006},
  publisher={Elsevier}
}

@article{berger2012simplified,
  title={A simplified h-box method for embedded boundary grids},
  author={Berger, Marsha and Helzel, Christiane},
  journal={SIAM Journal on Scientific Computing},
  volume={34},
  number={2},
  pages={A861--A888},
  year={2012},
  publisher={SIAM}
}

@Article{Johansson2013,
author="Johansson, August
and Larson, Mats G.",
title={A high order discontinuous {Galerkin Nitsche} method for elliptic problems with fictitious boundary},
journal="Numerische Mathematik",
year="2013",
month="Apr",
day="01",
volume="123",
number="4",
pages="607--628",
issn="0945-3245",
}

@article{kummer2017extended,
  title={{Extended discontinuous {Galerkin} methods for two-phase flows: the spatial discretization}},
  author={Kummer, Florian},
  journal={International Journal for Numerical Methods in Engineering},
  volume={109},
  number={2},
  pages={259--289},
  year={2017},
  publisher={Wiley Online Library}
}

@inproceedings{modisette2010toward,
  title={Toward a robust, higher-order cut-cell method for viscous flows},
  author={Modisette, James and Darmofal, David},
  booktitle={48th AIAA Aerospace Sciences Meeting Including the New Horizons Forum and Aerospace Exposition},
  pages={721},
  year={2010}
}

@article{muller2017high,
  title={A high-order discontinuous {Galerkin} method for compressible flows with immersed boundaries},
  author={M{\"u}ller, Bj{\"o}rn and Kr{\"a}mer-Eis, Stephan and Kummer, Florian and Oberlack, Martin},
  journal={International Journal for Numerical Methods in Engineering},
  volume={110},
  number={1},
  pages={3--30},
  year={2017},
  publisher={Wiley Online Library}
}

@article{Qin2013,
title={{A discontinuous Galerkin method for solutions of the Euler equations on Cartesian grids with embedded geometries}},
  author={Qin, Ruibin and Krivodonova, Lilia},
  journal={Journal of Computational Science},
  volume={4},
  number={1-2},
  pages={24--35},
  year={2013},
  publisher={Elsevier}
}

@article{gurkan2019stabilized,
  title={A stabilized cut discontinuous Galerkin framework for elliptic boundary value and interface problems},
  author={G{\"u}rkan, Ceren and Massing, Andr{\'e}},
  journal={Computer Methods in Applied Mechanics and Engineering},
  volume={348},
  pages={466--499},
  year={2019},
  publisher={Elsevier}
}

@article{gurkan2020stabilized,
  title={Stabilized Cut Discontinuous {Galerkin} Methods for Advection-Reaction Problems},
  author={G\"{u}rkan, Ceren and Sticko, Simon and Massing, Andr\'{e}},
  journal={SIAM Journal on Scientific Computing},
  volume={42},
  number={5},
  pages={A2620--A2654},
  year={2020},
  publisher={SIAM}
}

@article{burman2025cut,
  title={Cut finite element methods},
  author={Burman, Erik and Hansbo, Peter and Larson, Mats G and Zahedi, Sara},
  journal={Acta Numerica},
  volume={34},
  pages={1--121},
  year={2025},
  publisher={Cambridge University Press}
}

@article{gottlieb2001strong,
  title={Strong stability-preserving high-order time discretization methods},
  author={Gottlieb, Sigal and Shu, Chi-Wang and Tadmor, Eitan},
  journal={SIAM review},
  volume={43},
  number={1},
  pages={89--112},
  year={2001},
  publisher={SIAM}
}

@article{fu2021high,
author = {Fu, Pei and Kreiss, Gunilla},
title = {High Order Cut Discontinuous {Galerkin} Methods for Hyperbolic Conservation Laws in One Space Dimension},
journal = {SIAM Journal on Scientific Computing},
volume = {43},
number = {4},
pages = {A2404-A2424},
year = {2021},
doi = {10.1137/20M1349060}
}

@article{burman2010ghost,
  title={Ghost penalty},
  author={Burman, Erik},
  journal={Comptes Rendus Mathematique},
  volume={348},
  number={21-22},
  pages={1217--1220},
  year={2010},
  publisher={Elsevier}
}

@article{larson2021conservative,
  title={Conservative Discontinuous Cut Finite Element Methods},
  author={Mats G. Larson and Sara Zahedi},
  journal={arXiv preprint arXiv:2105.02202},
  year={2021}
}

@inproceedings{barth1989design,
  title={{The design and application of upwind schemes on unstructured meshes}},
  author={Barth, Timothy and Jespersen, Dennis},
  booktitle={27th Aerospace sciences meeting},
  pages={366},
  year={1989}
}

@article{fu2024bound,
  title={A bound preserving cut discontinuous Galerkin method for one dimensional hyperbolic conservation laws},
  author={Fu, Pei and Kreiss, Gunilla and Zahedi, Sara},
  journal={ESAIM: Mathematical Modelling and Numerical Analysis},
  volume={58},
  number={5},
  pages={1651--1680},
  year={2024},
  publisher={EDP Sciences}
}

@article{zalesak1979fully,
  title={Fully multidimensional flux-corrected transport algorithms for fluids},
  author={Zalesak, Steven T},
  journal={Journal of computational physics},
  volume={31},
  number={3},
  pages={335--362},
  year={1979},
  publisher={Elsevier}
}

@article{xiong2015high,
  title={High order maximum-principle-preserving discontinuous Galerkin method for convection-diffusion equations},
  author={Xiong, Tao and Qiu, Jing-Mei and Xu, Zhengfu},
  journal={SIAM Journal on Scientific Computing},
  volume={37},
  number={2},
  pages={A583--A608},
  year={2015},
  publisher={SIAM}
}

@article{xu2014parametrized,
  title={Parametrized maximum principle preserving flux limiters for high order schemes solving hyperbolic conservation laws: one-dimensional scalar problem},
  author={Xu, Zhengfu},
  journal={Mathematics of Computation},
  volume={83},
  number={289},
  pages={2213--2238},
  year={2014}
}

@article{CHRISTLIEB2015334,
title = {High order parametrized maximum-principle-preserving and positivity-preserving {WENO} schemes on unstructured meshes},
journal = {Journal of Computational Physics},
volume = {281},
pages = {334-351},
year = {2015},
issn = {0021-9991},
doi = {https://doi.org/10.1016/j.jcp.2014.10.029},
url = {https://www.sciencedirect.com/science/article/pii/S0021999114007116},
author = {Andrew J. Christlieb and Yuan Liu and Qi Tang and Zhengfu Xu}
}

@article{zhang2010maximum,
  title={On maximum-principle-satisfying high order schemes for scalar conservation laws},
  author={Zhang, Xiangxiong and Shu, Chi-Wang},
  journal={Journal of Computational Physics},
  volume={229},
  number={9},
  pages={3091--3120},
  year={2010},
  publisher={Elsevier}
}

@article{zhang2010positivity,
  title={{On positivity-preserving high order discontinuous Galerkin schemes for compressible Euler equations on rectangular meshes}},
  author={Zhang, Xiangxiong and Shu, Chi-Wang},
  journal={Journal of Computational Physics},
  volume={229},
  number={23},
  pages={8918--8934},
  year={2010},
  publisher={Elsevier}
}

@article{Zhang2011MaximumPrincipleSatisfyingAP,
  title={Maximum-Principle-Satisfying and Positivity-Preserving High Order Discontinuous {Galerkin} Schemes for Conservation Laws on Triangular Meshes},
  author={Xiangxiong Zhang and Yinhua Xia and Chi-Wang Shu},
  journal={Journal of Scientific Computing},
  year={2011},
  volume={50},
  pages={29 - 62},
  url={https://api.semanticscholar.org/CorpusID:17890908}
}

@article{10.1098/rspa.2011.0153,
    author = {Zhang, Xiangxiong and Shu, Chi-Wang},
    title = {Maximum-principle-satisfying and positivity-preserving high-order schemes for conservation laws: survey and new developments},
    journal = {Proceedings of the Royal Society A: Mathematical, Physical and Engineering Sciences},
    volume = {467},
    number = {2134},
    pages = {2752-2776},
    year = {2011},
    month = {05},
    issn = {1364-5021},
    doi = {10.1098/rspa.2011.0153},
    url = {https://doi.org/10.1098/rspa.2011.0153},
    eprint = {https://royalsocietypublishing.org/rspa/article-pdf/467/2134/2752/801033/rspa.2011.0153.pdf},
}

@article{liu1996nonoscillatory,
  title={Nonoscillatory high order accurate self-similar maximum principle satisfying shock capturing schemes I},
  author={Liu, Xu-Dong and Osher, Stanley},
  journal={SIAM Journal on Numerical Analysis},
  volume={33},
  number={2},
  pages={760--779},
  year={1996},
  publisher={SIAM}
}

@article{tan2010inverse,
  title={Inverse Lax-Wendroff procedure for numerical boundary conditions of conservation laws},
  author={Tan, Sirui and Shu, Chi-Wang},
  journal={Journal of Computational Physics},
  volume={229},
  number={21},
  pages={8144--8166},
  year={2010},
  publisher={Elsevier}
}

@article{may2022dod,
  title={DoD stabilization for non-linear hyperbolic conservation laws on cut cell meshes in one dimension},
  author={May, Sandra and Streitb{\"u}rger, Florian},
  journal={Applied Mathematics and Computation},
  volume={419},
  pages={126854},
  year={2022},
  publisher={Elsevier}
}

@article{engwer2020stabilized,
  title={A stabilized DG cut cell method for discretizing the linear transport equation},
  author={Engwer, Christian and May, Sandra and N{\"u}{\ss}ing, Andreas and Streitb{\"u}rger, Florian},
  journal={SIAM Journal on Scientific Computing},
  volume={42},
  number={6},
  pages={A3677--A3703},
  year={2020},
  publisher={SIAM}
}

@article{berger2021state,
  title={A state redistribution algorithm for finite volume schemes on cut cell meshes},
  author={Berger, Marsha and Giuliani, Andrew},
  journal={Journal of computational physics},
  volume={428},
  pages={109820},
  year={2021},
  publisher={Elsevier}
}

@article{giuliani2022two,
  title={A two-dimensional stabilized discontinuous Galerkin method on curvilinear embedded boundary grids},
  author={Giuliani, Andrew},
  journal={SIAM Journal on Scientific Computing},
  volume={44},
  number={1},
  pages={A389--A415},
  year={2022},
  publisher={SIAM}
}

@article{berger2003h,
  title={H-box methods for the approximation of hyperbolic conservation laws on irregular grids},
  author={Berger, Marsha J and Helzel, Christiane and LeVeque, Randall J},
  journal={SIAM Journal on Numerical Analysis},
  volume={41},
  number={3},
  pages={893--918},
  year={2003},
  publisher={SIAM}
}

@article{schoeder2020high,
  title={High-order cut discontinuous Galerkin methods with local time stepping for acoustics},
  author={Schoeder, Svenja and Sticko, Simon and Kreiss, Gunilla and Kronbichler, Martin},
  journal={International Journal for Numerical Methods in Engineering},
  volume={121},
  number={13},
  pages={2979--3003},
  year={2020},
  publisher={Wiley Online Library}
}

@article{spiteri2002new,
  title={A new class of optimal high-order strong-stability-preserving time discretization methods},
  author={Spiteri, Raymond J and Ruuth, Steven J},
  journal={SIAM Journal on Numerical Analysis},
  volume={40},
  number={2},
  pages={469--491},
  year={2002},
  publisher={SIAM}
}

@article{li2020discontinuous,
  title={Discontinuous Galerkin Methods for Nonlinear Scalar Conservation Laws: Generalized Local Lax--Friedrichs Numerical Fluxes},
  author={Li, Jia and Zhang, Dazhi and Meng, Xiong and Wu, Boying and Zhang, Qiang},
  journal={SIAM Journal on Numerical Analysis},
  volume={58},
  number={1},
  pages={1--20},
  year={2020},
  publisher={SIAM}
}
\end{document}